\documentclass[oneside, a4paper, 11pt]{amsart}

\usepackage{amsmath,amsthm,amssymb,bbm,comment,mathrsfs, tikz}

\usepackage[all]{xy}
\usepackage{thmtools}

\newtheorem{theorem}[subsubsection]{Theorem}
\newtheorem{lemma}[theorem]{Lemma}
\newtheorem{prop}[theorem]{Proposition}
\newtheorem{corollary}[subsubsection]{Corollary}
\newtheorem{conjecture}[theorem]{Conjecture}

\theoremstyle{definition}
\newtheorem{definition}[subsubsection]{Definition}

\newtheorem{remark}[theorem]{Remark}

\newtheorem{example}[subsubsection]{Example}

\newtheorem{question}[theorem]{Question}

\numberwithin{equation}{section}

\newcommand{\mC}{\mathbb{C}}

\newcommand{\mN}{\mathbb{N}}
\newcommand{\mZ}{\mathbb{Z}}
\newcommand{\mR}{\mathbb{R}}
\newcommand{\mNinf}{\mathbb{N}_{\infty}}
\newcommand{\bk}{\Bbbk}
\newcommand{\Spec}{\operatorname{Spec}}
\newcommand{\TSpec}{\operatorname{TSpec}}
\newcommand{\TIndec}{\operatorname{TIndec}}
\newcommand{\Irr}{\operatorname{Irr}}
\newcommand{\Ann}{\operatorname{Ann}}
\newcommand{\TAnn}{\operatorname{TAnn}}
\newcommand{\TSys}{\operatorname{TSys}}
\newcommand{\LAnn}{\operatorname{lAnn}}
\newcommand{\RAnn}{\operatorname{rAnn}}
\newcommand{\Res}{\operatorname{Res}}
\newcommand{\Ind}{\operatorname{Ind}}

\newcommand{\End}{\operatorname{End}}

\newcommand{\supp}{\operatorname{supp}}
\newcommand{\Sym}{\operatorname{Sym}}
\newcommand{\gr}{\operatorname{gr}}
\newcommand{\Rep}{\operatorname{Rep}}
\newcommand{\p}{\mathtt{p}}
\newcommand{\q}{\mathtt{q}}

\newcommand{\aaa}{\mathtt{a}}
\newcommand{\sss}{\mathtt{s}}
\newcommand{\ddd}{\mathtt{d}}
\newcommand{\gggr}{\mathtt{g}}

\newcommand{\cA}{\mathcal{A}}
\newcommand{\cC}{\mathcal{C}}
\newcommand{\cD}{\mathcal{D}}
\newcommand{\cT}{\mathcal{T}}

\newcommand{\bP}{\mathbf{P}}
\newcommand{\bB}{\mathbf{B}}
\newcommand{\bY}{\mathbf{Y}}
\newcommand{\mP}{\mathfrak{P}}

\newcommand{\Vecc}{\mathtt{Vec}}
\newcommand{\sVec}{\mathtt{sVec}}
\newcommand{\Ver}{\mathtt{Ver}}

\newcommand{\tto}{\twoheadrightarrow}

\newcommand{\unit}{\mathbf{1}}

\newcommand{\OB}{\mathcal{O}\mathcal{B}}
\newcommand{\cI}{\mathcal{I}}

\newcommand{\mF}{\mathbb{F}}
\begin{document}

\title[T-prime ideals]{A special class of prime ideals for infinite symmetric group algebras}

\author{Kevin Coulembier}
\address{School of Mathematics and Statistics, University of Sydney, NSW 2006, Australia}
\email{kevin.coulembier@sydney.edu.au}

\maketitle

\begin{abstract}
We identify an interesting special class of prime ideals in the finitary infinite symmetric group algebra. We show that the set of such ideals carries a semiring structure. Over the complex numbers, we establish a connection with spherical representations of (the Gelfand pair corresponding to) the infinite symmetric group. In positive characteristic, we investigate a close connection with the structure theory of symmetric tensor categories.
\end{abstract}


\section*{Introduction}

We consider the symmetric group $S_\infty$ of finitary permutations of a countable set, and study ideals in the goup algebra $\bk S_\infty$, for a field $\bk$. If $\bk$ has characteristic zero, the poset of (prime) ideals in $\bk S_\infty$ has a classical combinatorial description, see \cite{Br, FL, VK}. On the other hand, if $\bk$ has characteristic $p>0$, our undestanding of (prime) ideals in $\bk S_\infty$ is rather limited, see \cite{Za2} for some overview and motivation, despite an elegant representation-theoretic reformulation of part of the problem in \cite{Za}. A notable exception is the classification of maximal ideals, which has been completed in \cite{BK} for $p>2$. There are $p-1$ maximal ideals, which include the kernels of the trivial and sign representations, the two maximal ideals that exist in characteristic~0.

 In \cite{PolyFun}, a connection between the above results and (symmetric) tensor categories \cite{Del02, EGNO} was observed. To any object $X$ in a tensor category~$\cC$, one can associate an ideal $\TAnn_{\cC}(X)$ in $\bk S_\infty$ using the symmetric braiding. In characteristic 0, by \cite{Del02}, any tensor category of moderate growth (a condition which in the present context simply states that $\TAnn_{\cC}(X)$ is never zero) has a unique tensor functor to the category $\sVec$ of super vector spaces. Furthermore, the two annihilator ideals of the two simple objects in $\sVec$ are precisely the two maximal ideals in $\bk S_\infty$. By more recent results in \cite{CEO, EO, Os}, in characteristic $p>0$, any {\em Frobenius exact} tensor category of moderate growth has a unique tensor functor to the Verlinde category $\Ver_p$. Again, the annihilator ideals of the $p-1$ simple objects in $\Ver_p$ (for $p>2$ now) correspond precisely to the maximal ideals in $\bk S_\infty$. The categories $\Ver_p$ are part of a chain of tensor categories 
 $$\Ver_p\;\subset\; \Ver_{p^2}\;\subset\; \Ver_{p^3}\;\subset\;\cdots,$$
introduced in \cite{BEO, AbEnv}. It was conjectured in \cite[Conjecture~1.4]{BEO} that any tensor category of moderate growth has a unique tensor functor to $\Ver_{p^\infty}=\cup_n\Ver_{p^n}$. If $p=2$, it is known that there are at least 2 maximal ideals in $\bk S_\infty$, deviating from the pattern in \cite{BK}, and correspondingly, it was demonstrated in \cite{PolyFun} that these are the annihilator ideals of the two simple objects in $\Ver_4$. This deviates again from the pattern for $p>2$, but reinforces the impression that the connection with tensor categories is to be taken seriously.

The starting point of the current paper is therefore the question of to which extent the study of ideals in $\bk S_\infty$ and the structure theory of tensor categories in positive characteristic can benefit from one another. We can observe that not any ideal in $\bk S_\infty$ can be of the form $\TAnn_{\cC}(X)$. Indeed, the latter ideals satisfy the `primeness condition' that when $f\bullet g$ is in the ideal then one of the factors must be in the ideal, where $\bullet$ represents the `tensor' product $\bk S_m\times \bk S_n\to \bk S_{m+n}$. We call ideals that satisfy a linear version of this primeness condition {\bf T-prime}, and observe that these are in particular prime. We introduce two operations $\dagger$ and $\Join$ on ideals in $\bk S_\infty$ that make the set $\TSpec \bk S_\infty$ of T-prime ideals into a commutative semiring. It then follows that, for a tensor category $\cC$,
$$\TAnn_{\cC}: M^{\oplus}(\cC)\to \TSpec\bk S_\infty$$
is a semiring homomorphism from the split Grothendieck ring $M^{\oplus}(\cC)$. The above discussion in characteristic zero can now be summarised and enhanced to the statement that $\TAnn_{\cC}$ is an isomorphism (of semirings) if and only if $\cC=\sVec$. Moreover, in characteristic zero the T-prime ideals can be characterised in many equivalent ways, see Theorem~\ref{ThmClassT}, demonstrating that they form a canonical special class of prime ideals, worthy of further study. If $\bk=\mC$, we can additionally characterise them as the (left or right) annihilator ideals of the spherical $(S_\infty\times S_\infty, S_\infty)$-representations. For completeness, we determine the left annihilator ideals of the broader class of Olshanski-admissible representations of $S_\infty\times S_\infty$, as classified in \cite{Ok}, using more recent insight announced in \cite{VN}. Now the annihilator ideals can be arbitrary prime ideals, not necessarily T-prime ideals.

Of course, our main interest lies in positive characteristic $p>0$. We strengthen \cite[Conjecture~1.4]{BEO} into a conjecture that mainly states that
$$\TAnn_{\cC}: M^{\oplus}(\cC)\to \TSpec\bk S_\infty$$
is an isomorphism if and only if $\cC=\Ver_{p^\infty}$, and derive some evidence. We show for example that $\TAnn_{\cC}$ is injective for $\cC=\Ver_p$ and $\cC=\Ver_4$.

The paper is organised as follows. In Section~\ref{Prel}, we recall some preliminary results. In Section~\ref{SecT} we define T-prime ideals, the operations $\dagger$ and $\Join$, and some measures for the `size' of T-prime ideals. In Section~\ref{Sec0} we focus on characteristic zero and in Section~\ref{SecComp} we zoom in further on $\bk=\mC$. Finally, in Section~\ref{SecTC}, we initiate the deeper study of the connection between T-prime ideals and tensor categories in positive characteristic. In Section~\ref{SecDim} we explore annihilator ideals of the form $\TAnn$ obtained from more general symmetric monoidal categories, leading to the notion of the dimension of an ideal in $\bk S_\infty$, which is the image of an integer in $\bk$.


\section{Preliminaries}\label{Prel}

Let $\bk$ be an arbitrary field, unless further specified, and let $\mN$ denote the set of non-negative integers.

\subsection{The infinite symmetric group}

\subsubsection{} Denote by $S_n$, with $n\in\mZ_{>0}$, the permutation group of the set $\{1,2, \ldots,n\}$. This gives a canonical inclusion $S_n\subset S_{n+1}$.
We consider the infinite symmetric group and its group algebra
$$S_\infty:=\bigcup_n S_n\;=\;\varinjlim_n S_n\quad\mbox{and}\quad\bk S_\infty\;\simeq\;\varinjlim_n \bk S_n.$$

For the Young subgroup $S_m\times S_n$ of $S_{m+n}$ we use the corresponding algebra inclusion introduce the notation $-\bullet-$:
$$\bk S_m\otimes \bk S_n\;\to\; \bk S_{m+n},\quad f\otimes g\mapsto f\bullet g.$$
We will use the same notation $f\bullet g\in \bk S_\infty$ for $f\in \bk S_m\subset \bk S_\infty$ and $g\in \bk S_\infty$. If it is unclear in which $\bk S_m\subset \bk S_\infty$ we consider $f$, we will write more specifically $f\stackrel{m}{\bullet}g$.

\subsubsection{}
All ideals we consider are two-sided unless mentioned otherwise. An ideal $I<\bk S_\infty$ is determined by the ideals $I_n:=I\cap \bk S_n$ for all $n\in\mZ_{>0}$. Any system of ideals $I_n<\bk S_n$ with $I_{n}=I_{n+1}\cap \bk S_n$ determines uniquely a corresponding ideal $I=\cup_nI_n$. We let $\Spec \bk S_\infty$ be the topological space of prime ideals in $\bk S_\infty$, for the Jacobson topology.

For a ring $R$ and an $R$-module $M$, we denote by $\Ann_RM$, or simply $\Ann M$, its annihilator ideal in $R$.
Recall that for a subring $R\subset S$, with $S$ free as a right $R$-module, and a left $R$-module~$N$, we have
\begin{equation}\label{EqInd}
\Ann_S(S\otimes_RN)\;=\;\Ann_S(S/S\Ann_R(N)),
\end{equation}
so that the left-hand side only depends on $\Ann_R(N)$.

For two $\bk$-algebras $A,B$ with modules $M,N$, we have
\begin{equation}\label{AnnTensor}
\Ann_{A\otimes B}(M\boxtimes N)\;=\; \Ann_A(M)\otimes B+ A\otimes \Ann_B(N).
\end{equation}

\subsubsection{} The simple $\bk S_n$-modules are denoted by $D^\lambda$. Here $\lambda$ runs over all partitions of $n$ if $\mathrm{char}(\bk)=0$ (and $D^\lambda$ is the Specht module corresponding to $\lambda$) and over all $p$-regular partitions of $n$ if $\mathrm{char}(\bk)=p>0$.

For an appropriate partition $\lambda\vdash n$, we will use the symbol $e_\lambda$ to denote an arbitrary primitive idempotent in $\bk S_n$ corresponding to $D^\lambda$.

For a module $M$ over a finite dimensional associative $\bk$-algebra, we write $\Irr(M)$ for the set of simple modules that appear as simple constituents.

\subsubsection{}There are two notions of inductive systems relating to $ S_\infty$. 

By an {\bf inductive system}, we refer to an inductive system of~$S_\infty$ as in \cite[Definition~1.1]{Za}. This is a collection $\Phi^n$ of (isomorphism classes of) simple $\bk S_n$-modules for each $n\in\mZ_{>0}$, such that $V\in \Phi^n$ if and only if there is some $U\in \Phi^{n+1}$ so that $V$ is a composition factor of $\Res^{S_{n+1}}_{S_n}U$. Contrary to \cite{Za}, we allow the empty inductive system.

We use the term {\bf lifted inductive system} for an inductive system as in \cite[Definition~2.1.1]{PolyFun}. This is a collection $\bB^n$ of indecomposable $\bk S_n$-modules for each $n\in\mZ_{>0}$, such that $V\in \bB^n$ if and only if there is some $U\in \bB^{n+1}$ so that $V$ is a direct summand of $\Res^{S_{n+1}}_{S_n}U$.

If $\mathrm{char}(\bk)=0$, the two notions clearly coincide and we simply say `inductive system' for both. In general, for a lifted inductive system $\bB$ we can define the inductive system $K_0(\bB)$ such that $K_0(\bB)^n$ contains all simple modules that are constituents of the modules in $\bB^n$, so that we may view $\bB$ as a `lift' of $K_0(\bB)$.

An inductive system $\Phi$ is {\bf indecomposable} if we cannot write it as $\Phi_1\cup\Phi_2$ for proper subsystems $\Phi_i\subset\Phi$.

\subsubsection{}\label{SecPhiI} Consider an ideal $I<\bk S_\infty$. Then the following conditions are equivalent on a simple $\bk S_n$-module $D^\lambda$:
\begin{enumerate}
\item[(i)] $D^\lambda\in \Irr(\bk S_n/I_n)$;
\item[(ii)] $D^\lambda\in \Irr(\bk S_\infty/ I)$;
\item[(iii)] $e_\lambda\not\in I$.
\end{enumerate}
We can thus define an inductive system $\Phi(I)$, where
$\Phi(I)^n$ comprises those $D^\lambda$ that satisfy the conditions (i)-(iii). 
The assignment $I\mapsto \Phi(I)$ is, in a way that will be made more concrete in \cite{CEOq}, a very close analogue of the operation that take a tensor ideal in a Krull-Schmidt monoidal category and considers its underlying thick tensor ideal, see \cite{Selecta}. It is also an analogue of the connection between ideals in polynomial rings and algebraic sets, although as we will see the behaviour of prime ideals deviates.
\color{black}

Conversely, for an inductive system $\Phi$, we define  ideals
$$\hat{I}(\Phi)_n:=\bigcap_{V\in \Phi^n}\Ann_{\bk S_n}V\;<\; \bk S_n.$$
If $\mathrm{char}(\bk)=0$, or more generally if $\Phi$ is semisimple (meaning every $V\in \Phi^n$ restricts to a semisimple $S_m$-representation for any $m<n$), it follows that the collection $\{\hat{I}(\Phi)_n\mid n\}$ defines an ideal in $\bk S_\infty$, which we denote by $I(\Phi)$, so $I(\Phi)_n=\hat{I}(\Phi)_n$. In general, we define $I(\Phi)$ to be the maximal ideal $J$ in $\bk S_\infty$ that satisfies the following equivalent conditions:
\begin{enumerate}
\item[(i)] $J_n\subset\hat{I}(\Phi)_n$ for all $n$;
\item[(ii)] $e_\lambda\not\in J$ whenever $D^\lambda\in \Phi^n$;
\item[(iii)]  $\Phi\subset\Phi(J)$; 
\item[(iv)] $\Phi(J)=\Phi$.
\end{enumerate}
 The equivalence of the first three conditions is immediate. To obtain the equivalence between (iii) and (iv), it suffices to observe that there exists an ideal $K$ with $\Phi(K)=\Phi$. Indeed, existence of the minimal such ideal follows from the Jordan-H\"older theorem, as explained in the proof of \cite[Theorem~1.25]{Za}. \color{black}
It is also easy to give a direct description of $I(\Phi)$, namely
\begin{equation}\label{IPhi}I(\Phi)_n:=\bigcap_{m\ge n,V\in \Phi^m}\Ann_{\bk S_n}\Res^{S_m}_{S_n}V\;<\; \bk S_n.
\end{equation}

By the definition via (iv), we find, for any inductive system $\Phi$,
\begin{equation}
\label{eqPIP}
\Phi(I(\Phi))\;=\;\Phi.
\end{equation}
\color{black}

The following theorem is due to Zalesskii \cite{Za}, see also \cite{BK, Br, FL}.
Recall that an ideal is {\bf semiprimitive} if it is an intersection of primitive ideals. If the characteristic of $\bk$ is zero, then every ideal in $\bk S_\infty$ is semiprimitive, see for instance \cite[Lemma~1]{For}. 




\begin{theorem}[Zalesskii]\label{ThmZal}
The assignments
$I\mapsto \Phi(I)$ and $\Phi\mapsto I(\Phi)$ are (mutually inverse) inclusion reversing bijections between the set of semiprimitive ideals in $\bk S_\infty$ and the set of inductive systems.
 The bijection restricts to a bijection between the set of non-empty indecomposable inductive systems and semiprimitive prime ideals.
\end{theorem}
\begin{proof}
For fields of characteristic zero, this is \cite[Propositions~3.3 and~3.11]{Za}. For positive characteristic, it is \cite[Theorem~1.25]{Za}  and \cite[Proposition~8.5]{Za}.
\end{proof}

\begin{remark}
In characteristic zero one can show that every prime ideal in $\bk S_\infty$ is primitive. However, in positive characteristic, prime ideals (even T-prime ideals) need not be semiprimitive, see Remark~\ref{rem:primenotsp} below.
\end{remark}

\begin{example}
The zero ideal in $\bk S_\infty$ is prime, as follows for instance easily from Proposition~\ref{prop:Tpisp}(1) below. By the main theorem of \cite{For}, the zero ideal is also semiprimitive. Hence $0$ is the unique semiprimitive prime ideal corresponding to the inductive system containing all simple modules.
\end{example}
\color{black}

%
%
%
%

\subsection{Tensor categories}

\subsubsection{}
We follow the conventions from \cite{EGNO} regarding tensor categories, except that we only consider {\em symmetric} tensor categories and functors and hence leave `symmetric' out of the terminology. In particular, an essentially small $\bk$-linear symmetric category
$(\cC,\otimes,\unit)$ is a {\bf tensor category over $\bk$} if~$\cC$ is abelian with objects of finite length, $\bk\to\End_{\cC}(\unit)$ is an isomorphism and the monoidal category $(\cC,\otimes,\unit)$ is rigid.

The smallest tensor category over $\bk$ is the category $\Vecc$ of finite dimensional vector spaces over $\bk$. The tensor category $\sVec$ of
supervector spaces is the category of $\mZ/2$-graded vector spaces with braiding such that the braiding isomorphism on $\bar{\unit}\otimes\bar{\unit}$ for the non-trivial simple object $\bar{\unit}$ is $-1.$

For $X\in\cC$, the maximal quotient of the full tensor power $X^{\otimes n}$ that is invariant under the braiding action of the symmetric group is $\Sym^n X$.

\subsubsection{}Assume that $\bk$ is algebraically closed. A tensor category $\cC$ is {\bf incompressible} if every tensor functor from $\cC$ to a second tensor category is an equivalence onto a tensor subcategory, see~\cite{Incomp}.

A tensor category $\cC$ is of {\bf moderate growth} if for every $X\in\cC$ the (well-defined) limit
$$\mathsf{gd}(X)\;:=\; \lim_{n\to\infty}\sqrt[n]{\ell(X^{\otimes n})}\;\in\;\mR_{\ge 0}\cup\{\infty\}$$
is finite, see for instance~\cite[\S 4]{CEO}. 

By \cite[Theorem~5.2.1]{Incomp}, every tensor category of moderate growth admits a tensor functor to an incompressible tensor category (of moderate growth).
By classical results of Deligne \cite{Del02}, the only incompressible tensor categories in characteristic zero are $\Vecc$ and $\sVec$. If $\mathrm{char}(\bk)=p>0$, the only currently known incompressible tensor categories are the tensor subcategories of $\Ver_{p^\infty}$, see~\cite{BEO, AbEnv}.

{\bf Frobenius exact} tensor categories over fields of characteristic $p>0$ can be characterised in many equivalent ways, see for instance \cite{Tann, CEO, EO}. By the main result of \cite{CEO}, every Frobenius exact tensor category of moderate growth admits a tensor functor to $\Ver_p\subset\Ver_{p^\infty}$.


\section{T-prime ideals}\label{SecT}

In this section, we let $\bk$ be an arbitrary field.

\subsection{Definition}

\subsubsection{}
There are two obvious maximal ideals in $\bk S_\infty$, which coincide if $\mathrm{char}(\bk)=2$. Write $J_+$ and $J_-$ for the annihilator ideals of the trivial and sign module, or equivalently the kernels of the algebra morphisms
$$\bk S_\infty \to\bk,\; S_\infty\ni\sigma\mapsto 1,\qquad\mbox{or}\qquad \bk S_\infty \to\bk,\; S_\infty\ni\sigma\mapsto (-1)^{|\sigma|}.$$

\begin{lemma}\label{lem:totprim}
The only totally prime ideals in $\bk S_\infty$ are $J_+$ and $J_-$.
\end{lemma}
\begin{proof}
A totally prime in a ring restricts to a totally prime ideal in any subring. In $\bk S_n$, the only prime ideals are primitive and their quotients are $\End_{\bk}(V)$ for the corresponding simple $V$. Hence, the only totally prime ideals in $\bk S_n$ are the ideals of the one-dimensional representations, the trivial and the sign representation. 
\end{proof}

We introduce the following term for a weaker notion; the `T' stands for `totally', but also for `tensor'. We will see that the following two definitions coincide in characteristic zero. We currently also have no examples to show the two definitions differ in positive characteristic.\color{black}

\begin{definition}\label{DefT}
\begin{enumerate}
\item An ideal $I<\bk S_\infty$ is  {\bf weakly T-prime} \color{black} if for all $m,n\in \mZ_{>0}$, $f\in\bk S_m$ and $g\in \bk S_n$ with
$f\bullet g\in I_{m+n}$,
we have $f\in I$ or $g\in I$.
\item  An ideal $I<\bk S_\infty$ is {\bf T-prime} if for all $m,n\in \mZ_{>0}$, the following equivalent conditions are satisfied:
\begin{enumerate}
\item For each $m,n\in \mZ_{>0}$ we have
$$I_{m+n}\cap(\bk S_m\otimes \bk S_n)\;=\; I_m\otimes \bk S_n+\bk S_m\otimes I_n;$$
\item For each $m,n\in \mZ_{>0}$ we have
$$\Ann_{\bk S_m\otimes \bk S_n}\bk S_{m+n}/I_{m+n}\;=\; \Ann_{\bk S_m\otimes \bk S_n}(\bk S_m/I_m\boxtimes \bk S_n/I_n).$$
\end{enumerate}
\end{enumerate}

\end{definition}

It is obvious (even without Lemma~\ref{lem:totprim}) that a totally prime ideal must be weakly T-prime. On the other hand, we have the following result.

\begin{prop}\label{prop:Tpisp}
\begin{enumerate}
\item A weakly T-prime ideal is either prime or $\bk S_\infty$.
\item Any maximal ideal in  $\bk S_\infty$ is weakly T-prime.
\end{enumerate}

\end{prop}
\begin{proof} First we prove part (1).
If an ideal $I<\bk S_\infty$ is not prime, then there are $f\in \bk S_m$ and $g\in \bk S_n$, both not in $I$, for which $(f)(g)\subset I$. Then
$$f\bullet g\;\in\; (f)(g)\;\subset\; I,$$
showing that $I$ is not T-prime.

For part (2), let $I<\bk S_\infty$ be a maximal ideal and $f\in\bk S_m$ with $f\not\in I$. Hence we can write
$$1=afb+x$$
for $a,b\in \bk S_\infty$ and $x\in I$.
Now consider an arbitrary $g\in \bk S_n$ with $f\bullet g\in I$. Then, for sufficiently large $r$ we can write
$$1\stackrel{r}{\bullet} g\;=\;  (a\stackrel{r}{\bullet}1) (f\stackrel{r}{\bullet} g) (b\stackrel{r}{\bullet}1)+x\stackrel{r}{\bullet} g.$$
Both terms on the right are in $I$, so $g\in I$.
\end{proof}

\subsubsection{Convention} It will be convenient to {\em allow the identity ideal to be T-prime}, as done in Definition~\ref{DefT}. Conversely, and contrary to Definition~\ref{DefT}, from now on we {\em will assume that a T-prime ideal is not zero} unless explicitly stated otherwise.

\subsubsection{} We introduce the class of (indecomposable) inductive systems that link to T-prime ideals under the correspondence in Theorem~\ref{ThmZal}. We call an inductive system $\Psi$ {\bf T-indecomposable} if for every $D^\lambda\in\Phi_m$ and $D^\mu\in\Phi_n$, there exists $D^\kappa\in\Phi_{m+n}$ with
$$\left[\Res^{S_{m+n}}_{S_m\times S_n} D^\kappa \,:\, D^\lambda\boxtimes D^\mu\right]\;\not=0.$$
One directly verifies that T-indecomposable systems are indecomposable.  We denote the set of T-indecomposable systems as $\TIndec\bk S_\infty$ and again, unless explicitly stated otherwise, we do not consider the T-indecomposable inductive system of all simple modules.

Correspondingly, we denote by $\TSpec \bk S_\infty$ the topological space of T-prime ideals in $\bk S_\infty$, including the identity ideal but excluding the zero ideal. Strictly speaking this is thus not a subspace of $\Spec \bk S_\infty$.

\begin{prop}\label{LemDefT}\label{Prop:MIndec}
 
\begin{enumerate}
\item If an ideal $I<\bk S_\infty$ is weakly T-prime, then $\Phi(I)$ is T-indecomposable.
\item If $\Phi$ is T-indecomposable, then $I(\Phi)$ is T-prime.
\item If $I$ is semiprimitive and weakly T-prime, then it is T-prime.
\end{enumerate}
\noindent In particular, the assignment $\Phi$ restricts to a surjective map
\[\TSpec\bk S_\infty\;\stackrel{\Phi}{\tto}\;\TIndec\bk S_\infty,\]
and to a bijection between T-indecomposable inductive systems and semiprimitive T-prime ideals.
\end{prop}
\begin{proof}
First we deal with (1). Suppose that $I$ is a weakly T-prime ideal. If $D^\lambda\in \Phi(I)_m$ and $D^\mu\in\Phi(I)_n$, then $e_\lambda,e_\mu\not\in I$ and hence $e_\lambda\bullet e_\mu\not\in I$. This means that $D^\kappa \in \Phi(I)_{m+n}$ for some $D^\kappa$ appearing in the top of the projective $\bk S_{m+n} (e_\lambda\bullet e_\mu) $. The latter condition on $\kappa$ is equivalent to
$$\left[\Res^{S_{m+n}}_{S_m\times S_n} D^\kappa \,:\, D^\lambda\boxtimes D^\mu\right]\;\not=0.$$
This shows that $\Phi(I)$ is T-indecomposable.

Now we prove (2). For any ideal $I$, the right-hand side of the equation in Definition~\ref{DefT}(2)(a) is included in the left-hand side. For an inductive system $\Phi$, using \eqref{IPhi}, we can write the left-hand side for $I=I(\Phi)$ as
$$\bigcap_{r\ge m+n,W\in \Phi^r}\Ann_{\bk S_m\otimes \bk S_n} \Res^{S_r}_{S_m\times S_n}W.$$
Similarly, using \eqref{AnnTensor}, the right-hand side is
$$\bigcap_{d\ge m,t\ge n, U\in \Phi^d, V\in \Phi^t}\Ann_{\bk S_m\otimes \bk S_n} \Res^{S_d\times S_t}_{S_m\times S_n}U\boxtimes V.$$
Now, if $\Phi$ is T-indecomposable, any $U\boxtimes V$ as above is a constituent of some $W\in\Phi^{d+t}$ restricted to $S_d\times S_t$ and in particular
$$\Ann_{\bk S_m\otimes \bk S_n} \Res^{S_{d+t}}_{S_m\times S_n}W\;\subset\;\Ann_{\bk S_m\otimes \bk S_n} \Res^{S_d\times S_t}_{S_m\times S_n}U\boxtimes V,$$
where we point out that the different (conjugate) ways of interpreting $S_{m}\times S_n< S_{d+t}$ that appear above lead to the same annihilator. This observation thus implies that under our assumptions the left-hand side of \ref{DefT}(2)(a) is contained in the right-hand side and thus equal.

Claim (3) follows from the combination of (1) and (2), since $I=I(\Phi(I))$ for semiprimitive ideals $I$, see Theorem~\ref{ThmZal}.

The final sentences follow from equation~\eqref{eqPIP}.
\end{proof}

%
%

Since in characteristic zero all ideals in $\bk S_\infty$ are semiprimitive, Proposition~\ref{LemDefT} simplifies as follows.
\begin{corollary}
If $\mathrm{char}(\bk)=0$, then an ideal in $\bk S_\infty$ is weakly T-prime if and only if it is T-prime, and T-prime ideals are in bijection with T-indecomposable inductive systems.
\end{corollary}

The following consequence is new only for $p=2$, due to the classification for $p>2$ in \cite{BK}.

\begin{corollary}
\begin{enumerate}
\item Every maximal ideal in $\bk S_\infty$ is T-prime.
\item Every minimal (non-empty) inductive system is T-indecomposable.
\end{enumerate}

\end{corollary}
\begin{proof}
Part (1) follows from the combination of Proposition~\ref{prop:Tpisp}(2) and Proposition~\ref{LemDefT}(3).

Part (2) follows from the combination of Proposition~\ref{prop:Tpisp}(2) and Proposition~\ref{LemDefT}(1). We can also argue directly as follows.

Let $\Phi$ be a minimal inductive system. For a fixed $D^\lambda\in\Phi^m$ in the system consider the family of $D^\mu\in\Phi$ for which $D^\lambda\boxtimes D^\mu$ is a constituent of the restriction of some $D^\kappa\in\Phi$. By construction, this family is closed under taking the simple constituents of restrictions. By considering the commutative square of restriction functors
\[\xymatrix{
\Rep S_{m+n+1}\ar[rr]\ar[d]&& \Rep S_{m}\times S_{n+1}\ar[d]\\
\Rep S_{m+n}\ar[rr]&& \Rep S_{m}\times S_n
}\]
it follows that this family actually forms an inductive system. It is non-empty since the trivial representation of $S_1$ belongs to it. Hence the family is the entire system, by minimality.
\end{proof}

\color{black}

\subsection{Operations}
We introduce two binary operations on ideals in $\bk S_\infty$.

\begin{definition}
For ideals $I,J$ in $\bk S_\infty$, the ideal $I\Join J$ is defined as the preimage of $$I\otimes \bk S_\infty \;+\; \bk S_\infty\otimes J$$
under the comultiplication algebra morphism
$\Delta:\bk S_\infty\to\bk S_\infty \otimes \bk S_\infty$.
\end{definition}
By cocommutativity, we have $I\Join J=J\Join I$, and similarly coassociativity implies associativity of $\Join$. We also have
$$(I\Join J)_n\,= \,\Ann_{\bk S_n}\left(\bk S_n/I_n\otimes \bk S_n/J_n\right)\,=\, \Ann_{\bk S_n}\left(V\otimes W\right),$$
for any $\bk S_n$-modules $V,W$ with annihilator ideals $I_n,J_n$.

\begin{example}\label{ExPlus}
\begin{enumerate}
\item The augmentation ideal $J_+$ is the identity element for the operation $\Join$, meaning that for every ideal $I$ in $\bk S_\infty$,
$$I\Join J_+\;=\; I\;=\; J_+\Join I.$$
\item The identity ideal is the zero element for the operation $\Join$, meaning 
$$I\Join \bk S_\infty\;=\; \bk S_\infty\;=\; \bk S_\infty\Join I.$$
\item For two $\bk S_\infty$-modules, we have
$$\Ann (M\otimes N)\;=\; \Ann(M)\Join \Ann(N),$$
for example $J_-\Join J_-=J_+$.
\end{enumerate}
\end{example}


\begin{definition}
For ideals $I,J$ in $\bk S_\infty$, the ideal $I\dagger J$ is given by
\begin{eqnarray*}
(I\dagger J)_n&=& \bigcap_{i=0}^n \Ann_{\bk S_n} (\bk S_{n}/\bk S_n(I_i\otimes \bk S_{n-i}\,+\, \bk S_i\otimes J_{n-i}))\\
&=& \bigcap_{i=0}^n \Ann_{\bk S_n} \Ind^{S_n}_{S_i\times S_{n-i}}(\bk S_i/I_i\boxtimes \bk S_{n-i}/J_{n-i})\\
&=& \bigcap_{i=0}^n \Ann_{\bk S_n} \Ind^{S_n}_{S_i\times S_{n-i}}(V_i\boxtimes W_{n-i}),
\end{eqnarray*}
for any $\bk S_i$-module $V_i$ with annihilator $I_i$ and any $\bk S_{n-i}$-module $W_{n-i}$ with annihilator $J_{n-i}$.
\end{definition}
That this chain of ideals actually defines an ideal in $\bk S_\infty$ follows from Mackey's formula. By definition $I\dagger J=J\dagger I$, and also $\dagger$ is associative.
\begin{example}\label{ExDagger}
\begin{enumerate}
\item The identity ideal $\bk S_\infty$ is the identity element for the operation  $\dagger$, meaning
$$I\dagger \bk S_\infty\;=\; I\;=\; \bk S_\infty \dagger I.$$
\item For $m,n\in\mN$, we define the ideal 
$$\mP_{m,n}\,:= \,J_+^{\dagger m}\dagger J_-^{\dagger n},$$
so $\mP_{0,0}=\bk S_\infty$ by convention. Then $(\mP_{m,n})_i$ is the kernel of the permutation action
$$\bk S_i\;\to\; \End(V^{\otimes i}),$$
for $V$ a supervector space of dimension $m|n$. This follows directly, or from Example~\ref{ExsVec} below. We take the convention that for $\mathrm{char}(\bk)=2$, we have $\mP_{m,n}=J^{\dagger (m+n)}_+$.
\end{enumerate}

\end{example}

\begin{lemma}
For all ideals $I,J<\bk S_\infty$ and $P\in \TSpec \bk S_\infty$, we have
$$(I\dagger J)\Join P\;=\; (I\Join P)\dagger (J\Join P).$$
\end{lemma}
\begin{proof}
We can rewrite $((I\dagger J)\Join P)_n$ as
$$\bigcap_{i=0}^n\Ann\left(\Ind^{S_n}_{S_i\times S_{n-i}}\left((\bk S_{i}/I_i\boxtimes \bk S_{n-i}/J_{n-i}) \otimes \Res^{S_n}_{S_i\times S_{n-i}}\bk S_n/P_n\right)\right),$$
whereas $((I\Join P)\dagger (J\Join P))_n$ becomes
$$\bigcap_{i=0}^n\Ann\left(\Ind^{S_n}_{S_i\times S_{n-i}}\left((\bk S_i/I_i\otimes \bk S_i/P_i)\boxtimes (\bk S_{n-i}/J_{n-i}\otimes \bk S_{n-i}/P_{n-i})\right)\right).$$
The conclusion thus follows from condition \ref{DefT}(2)(b).
\end{proof}

\begin{theorem}
The set $\TSpec \bk S_\infty$ equipped with addition $\dagger$ and multiplication $\Join$ is a commutative semiring.
\end{theorem}
\begin{proof}
The only thing that remains to be proved is that $\dagger$ and $\Join$ restrict to operations
$$\TSpec\bk S_\infty\;\times\;\TSpec \bk S_\infty\;\to\; \TSpec \bk S_\infty.$$

The fact that $\Join$ preserves T-prime ideals can be reduced to the observation that the annihilator ideal in $\bk S_m\otimes \bk S_n$ of
$$\bk S_{m+n}/I_{m+n}\;\otimes\; \bk S_{m+n}/J_{m+n}$$
coincides with that of
$$ (\bk S_{m}/I_{m}\;\otimes\; \bk S_{m}/J_{m})\;\boxtimes\; (\bk S_{n}/I_{n}\;\otimes\; \bk S_{n}/J_{n}),$$
assuming that $I,J$ are T-prime.

The case $\dagger$ follows similarly, by also applying Mackey's formula.
\end{proof}

\begin{remark}
Already for $\mathrm{char}(\bk)=0$, the semiring $\TSpec \bk S_\infty$ is not topological for the Jacobson topology, as follows from the explicit description in the Section~\ref{Sec0}.
\end{remark}

\begin{remark}
For ideals $I,J$ it follows quickly that $\Phi(I\dagger J)$ and $\Phi(I\Join J)$ depend only on $\Phi(I)$ and $\Phi(J)$, so that $\TIndec\bk S_\infty$ becomes a quotient semiring of $\TSpec\bk S_\infty$. The operation $\Join$ sends two inductive systems $\Phi_1,\Phi_2$ to the inductive system containing all simple constituents of modules $D_1\otimes D_2$ with $D_1\in\Phi^n_1$, $D_2\in\Phi^n_2$.
\end{remark}

One can derive from Remark~\ref{rem:primenotsp} below, using \eqref{eq:Ver4}, that $I\Join J$ need not be semiprimitive when $I,J$ are.
\begin{question}
Is $I\dagger J$ semiprimitive whenever $I$ and $J$ are semiprimitive ideals (in $\TSpec$)?
\end{question}

\subsection{Growth rates of ideals}
\label{SecGrowth}

\subsubsection{} As a rough estimate of the size of a non-zero ideal $I<\bk S_\infty$, we set
$$\ddd(I)\;:=\; \max\{n\in\mZ_{>0}\mid I_n=0\}\;\in\;\mN,$$
with the understanding that the maximum of an empty set is $0$. Hence $\ddd(I)=0$ if and only if $I=\bk S_\infty$.
The following lemma is immediate.
\begin{lemma}\label{LemdAdd}
For non-zero ideals $I,J$ in $\bk S_\infty$, we have
$$\ddd(I\dagger J)\;\ge\; \ddd(I)+\ddd(J).$$
\end{lemma}

\begin{corollary}
For a non-zero proper ideal $I<\bk S_\infty$ and $m,n\in \mN$, we have $I^{\dagger m}=I^{\dagger n}$ if and only if $m=n$.
\end{corollary}

\subsubsection{} For T-prime ideals, we can define a more refined measure of the size. By definition, for $I\in \TSpec \bk S_\infty$, the canonical morphism
$$\bk S_m/I_m\otimes \bk S_n/I_n\;\to\; \bk S_{m+n}/I_{m+n}$$
is an inclusion, so
$$\dim_{\bk}(\bk S_m/I_m)\dim_{\bk}(\bk S_m/I_m)\;\le\; \dim_{\bk}(\bk S_{m+n}/I_{m+n}).$$
Hence, setting $g_n(I)=\dim (\bk S_n/I_n)$, by Fekete's lemma in~\cite[Lemma~1.6.3]{Be}, there is a well-defined limit
$$\gggr(I)\;=\; \lim_{n\to\infty}\sqrt[2n]{g_n(I)}\;=\; \sup\{\sqrt[2n]{g_n(I)}\mid n>0\}\;\,\in\;\mR_{\ge 0}\cup\{\infty\}.$$
As with $\ddd(I)$, higher $\gggr(I)$ implies smaller $I$. For example $\gggr(\bk S_\infty)=0$ and $\gggr(I)\ge 1$ otherwise.

It will follow from closer analysis that, if we exclude the case $I=0$, we have $\gggr(I)\in\mR_{\ge0}$ for all $I\in\TSpec\bk S_\infty$, see Corollary~\ref{Cor0} and Lemma~\ref{lemfin}.

\begin{lemma}\label{Lemgdagger}
For $I,J\in \TSpec\bk S_\infty$, we have
\begin{enumerate}
\item $\max\{\gggr(I),\gggr(J)\}\;\le\; \gggr(I\dagger J)\;\le\; \gggr(I)+\gggr(J)$;
\item $\gggr(I\Join J)\;\le\; \gggr(I)\gggr(J)$.
\end{enumerate}
\end{lemma}
\begin{proof}
The first inequality in (1) follows since $I\dagger J\subset I$. For the second inequality, by definition, we have an exact sequence
$$0\to (I\dagger J)_{n}\to \bk S_n\to \bigoplus_{i=0}^n \End_{\bk}\left(\Ind^{S_n}_{S_i\times S_{n-i}} (\bk S_i/I_i\boxtimes \bk S_{n-i}/J_{n-i})\right),$$
where the right arrow represents the action of $\bk S_n$. This action commutes with the obvious right $\bk S_i\otimes \bk S_{n-i}$-action on the induced module, so that the exact sequence can be adjusted to get the estimate
$$g_n(I\dagger J)\;\le\;\sum_{i=0}^n\binom{n}{i}^2 g_i(I)g_{n-i}(J)$$
and hence
$$\sqrt{g_n(I\dagger J)}\;\le\;\sum_{i=0}^n\binom{n}{i} \sqrt{g_i(I)}\sqrt{g_{n-i}(J)}.$$
The result then follows from a standard argument, see for instance the proof of \cite[Lemma~4.9(2)]{CEO}.

Part (2) follows more directly from the exact sequence
$$0\to (I\Join J)_n\to \bk S_n\to \bk S_n/I_n\otimes \bk S_n/J_n$$
where the right arrow is comultiplication followed by quotient map. 
\end{proof}

At least in some cases, the subadditivity of Lemma~\ref{Lemgdagger}(1) is actually additivity:

\begin{lemma}\label{Lemmag}
\begin{enumerate}
\item We have $\gggr(\mP_{m,0})=\gggr(J_+^{m\dagger})=m$.
\item If $\mathrm{char}(\bk)=0$, then $\gggr(\mP_{m,n})=m+n$.
\end{enumerate}
\end{lemma}
\begin{proof}
We start with part (2). By Schur-Weyl duality, see for instance \cite{Berele}, and Example~\ref{ExDagger}(2), we have
$$\gggr(\mP_{m,n})\;=\;\lim_{i\to\infty}\sqrt[2i]{\dim_{\bk}\End_{GL(V)}(V^{\otimes i})},$$
for $V$ a supervector space of dimension $m|n$.
Now by \cite{Berele}, the $GL(V)$-representation $V^{\otimes i}$ is semisimple and the number of isomorphism classes of simple modules that can appear as direct summand in $V^{\otimes i}$ grows only polynomially in $i$, as it is bounded by $(i)^{m+n}$. On the other hand, with $d_i$ the length of $V^{\otimes i}$, we have 
$$\lim_{i\to\infty}\sqrt[i]{d_i}\;=\; m+n,$$
see for instance \cite[Proposition~3.1]{COT}. There must be a simple module that occurs at least $d_i/i^{m+n}$ times as direct summand in $V^{\otimes i}$, giving a lower bound for the dimension of $\End_{GL(V)}(V^{\otimes i})$. Combining all the above observations, we find that
$$m+n\;=\;\lim_{i\to\infty}\sqrt[2i]{(d_i/i^{m+n})^2}\;\le\;\gggr(\mP_{m,n}),$$
so the conclusion follows from Lemma~\ref{Lemgdagger}(1).

Part (1) follows from part (2), Schur-Weyl duality in positive characteristic and independence of $\dim_{\bk}\End_{GL(V)}(V^{\otimes i})$ on the characteristic (a well-known consequence of the fact that $V^{\otimes i}$ is a tilting module in a highest weight category with standard modules that have the same character in every characteristic).
\end{proof}

\begin{remark}
As observed in \cite{CEKO}, for $I=\mP_{m,n}$, the dimension of $\bk S_i/I_i$ s lower in positive characteristic than in characteristic zero, so a priori the inequality $\gggr(\mP_{m,n})\le m+n$ from Lemma~\ref{Lemgdagger} might be strict. Note however that $\gggr(\mP_{1,1})=2$ by \cite[Theorem~3.3.1]{CEKO}.
\end{remark}

\subsection{Tensor categories}
Let $\cC$ be a tensor category over $\bk$ and $X\in \cC$.

\subsubsection{} Following \cite[\S 6]{PolyFun}, to $X$ we associate the ideal
$$\TAnn(X)\;<\; \bk S_\infty.$$ Concretely, $\TAnn(X)$ is the kernel of the algebra morphism 
$$\bk S_\infty\;\to\; \varinjlim_{n}\End(X^{\otimes n})$$
coming from the braiding on $\cC$. Equivalently, $\TAnn(X)_n$ is the kernel of
$\bk S_n\to \End(X^{\otimes n}),$
which is the common kernel of all modules in $\mathbf{B}_X^n$, for $\mathbf{B}_X$ the lifted inductive system associated to $X$ in \cite[\S 3]{PolyFun}. By construction $\TAnn(FX)=\TAnn(X)$ for any tensor functor $F$ (which are exact and faithful by our conventions). 
We also let $\TSys(X)=\Phi(\TAnn(X))$ be the corresponding inductive system.
\color{black}

As proved in \cite[4.11 and 4.12]{CEO}, the category $\cC$ is `of moderate growth' if and only if $\TAnn(X)\not=0$ for all $X\in\cC$.

\subsubsection{} For a tensor category, we consider the split Grothendieck monoid $M^{\oplus}(\cC)$, which is isomorphic to the free monoid generated by the set of isomorphism classes of indecomposable objects in $\cC$. This is a commutative semiring, with product induced from the tensor product. The Grothendieck monoid $M(\cC)$ is isomorphic to the free monoid generated by the set of simple objects of $\cC$ and we have an obvious surjection $M^{\oplus}(\cC)\tto M(\cC)$.\color{black}

\begin{prop}\label{Lem:Tens} Let $\cC$ be a tensor category of moderate growth.
The association $X\mapsto \TAnn(X)$ induces a semiring morphism
$$\TAnn_{\cC}:\;M^{\oplus}(\cC)\;\to\; \TSpec\bk S_\infty,$$
and fits into a commutative square of semiring morphisms
\[\xymatrix{
M^{\oplus}(\cC)\ar[rr]^-{\TAnn}\ar@{->>}[d]\ar[drr]_{\TSys}&& \TSpec\bk S_\infty\ar@{->>}[d]^{\Phi}\\
M(\cC)\ar[rr]&&\TIndec\bk S_\infty.
}\]
\end{prop}

\begin{proof}
Firstly, $\TAnn(X)$ is T-prime, as follows for instance from (the proof of) \cite[Proposition~2.1.3]{CEOq}. It thus suffices to observe the identities
\begin{eqnarray*}
\TAnn(X\oplus Y)&=& \TAnn(X)\dagger \TAnn(Y),\\
\TAnn(X\otimes Y)&=&\TAnn(X)\Join \TAnn(Y),\\
\TAnn(0)&=& \bk S_\infty,\qquad \mbox{and}\quad \TAnn(\unit)\;=\; J_+,
\end{eqnarray*}
which follow by direct verification.

 To conclude the proof, we need to show that $\TSys(X)$ only depends on the simple constituents (with multiplicities) of $X$. This follows from the principle that if we consider a filtration on $X$, then the braiding on the associated graded $\bk S_n\to \End((\gr X)^{\otimes n})$ factors via $\bk S_n\to \End(X^{\otimes n})$, see the proof of \cite[Lemma~4.9]{CEO}. From this it follows that $\TAnn(\gr X)_n/\TAnn( X)_n$ is a nilpotent ideal, and hence $\TSys(X)=\TSys(\gr X)$.
\end{proof}

\begin{example}\label{ExsVec}
If $\cC=\sVec$ and $V\in\cC$ a supervector space of dimension $m|n$, then $\TAnn(V)=\mP_{m,n}$, as now follows from the definition in Example~\ref{ExDagger}(2).
\end{example}

\begin{prop}
The inductive system $\TSys(X)$ is equal to $K_0(\mathbf{B}_X)$.
\end{prop}
\begin{proof}
This follows from the observations in \ref{SecPhiI}.
\end{proof}

\begin{remark}\label{ggd}
We have
$$g_n(\Ann(X))\;\le\; \dim_{\bk}\End(X^{\otimes n})\;\le\; \ell(X^{\otimes n})^2$$
and therefore
$$\gggr(\Ann(X))\;\le\; \mathtt{gd}(X).$$
\end{remark}

\begin{remark}\label{rem:primenotsp}
Proposition~\ref{Lem:Tens} provides a setting to show that in positive characteristic $\TSpec\bk S_\infty\tto \TIndec\bk S_\infty$ is not a bijection, or in other words, there exist T-prime ideals in $\bk S_\infty$ that are not semiprimitive. Concretely, in the setting of Theorem~\ref{ThmVer}(1) below, we observe that $\TAnn(P)\not=\TAnn(\gr P)$, for $P$ the indecomposable projective object $P\in \Ver_4^+$ with associated graded $\gr P=\unit\oplus \unit$. This shows that the T-prime ideal $\TAnn(P)$ is not semiprimitive.
\end{remark}

\color{black}

\section{Characteristic zero}\label{Sec0}

In this section, we let $\bk$ be a field of characteristic 0.

\subsection{Classification of prime ideals}

We recall the classification of (prime) ideals in $\bk S_\infty$, as first obtained in~\cite{FL} or \cite{Br}, in a slightly different language.

\subsubsection{} Set $\mNinf=\mN\cup\{+\infty\}$. We define a {\bf partition} to be a non-increasing function 
$$\lambda:\;\mZ_{>0}\to \mNinf,\quad i\mapsto \lambda_i.$$
We set
$$|\lambda|\;=\;\sum_{i}\lambda_i\;\in\; \mNinf.$$
If $n=|\lambda|\in\mNinf$, then we will write $\lambda\vdash n$ and call $\lambda$ an $n$-partition.
We also denote by $\bP$ the set of all partitions, partially ordered for the inclusion order $\lambda\subset \mu$ (meaning $\lambda_i\le\mu_i$ for all $i$), with sub-poset $\bP_\infty$ of $\infty$-partitions.

We can reinterpret the complement of $\bP_\infty$ in $\bP$ as the Young graph $\bY$, an $\mN$-graded graph where $\bY_n$ consists of the $n$-partitions. The edges are the covers in the inclusion order, see \cite[Definition~3.2]{BO}.

\subsubsection{} A path $\p$ in the Young graph $\bY$ (called an infinite tableau in \cite{VK}) is a choice $\p(n)\in \bY_n$ for all $n\in\mZ_{>0}$ so that $\p(n)\subset \p(n+1)$.

 Let $\p$ be such a path in $\bY$. Since, for a fixed $i\in\mZ_{>0}$, we always have $\p(n)_i\le \p(n+1)_i$, we can define the `shape of the tableau $\p$'
$$\p(\infty)_i\:=\;\lim_{n\to\infty}\p(n)_i\;=\; \sup\{\p(n)_i\mid n\in\mZ_{>0}\}\;\in\; \mNinf.$$
This produces $\p(\infty)\in \bP_\infty$. Clearly any $\infty$-partitions can be written as such a limit of a path in the Young graph.

\begin{theorem}\label{Prop:0}
\begin{enumerate}
\item There is an inclusion order reversing bijection from $\bP_\infty$ to $\Spec \bk S_\infty$, given by $\lambda\mapsto I_\lambda$, where for all $n\in\mZ_{>0}$
$$I_\lambda\cap \bk S_n\;=\; \bigcap_{\mu\vdash n\mid \mu\subset \lambda}\Ann(D^\mu).$$
\item Every ideal in $\bk S_\infty$ is a finite intersections of prime ideals.
 \end{enumerate}
\end{theorem}

We will use the following technical lemma in the proof.

\begin{lemma}\label{Lem:Incomp}
Every set of incomparable partitions (for the inclusion order) is finite.
\end{lemma}
\begin{proof}
To find a contradiction, let $S$ be an infinite set of incomparable partitions.
We distinguish into two cases. 

Case (I): There is no bound on the number of labels that can be $\infty$ for elements of $S$. In that case we can replace $S$ with a subset that contains, for each $i\in\mN$ at most one element of the form $(\infty^i,a_{i+1},\cdots)$ with $a_{i+1}\in\mN$. Purely for notational simplicity we assume that there is {\it precisely} one, for each $i\in\mN$.
Then
$$S\;=\;\{(\infty^i, a_{i,i+1},a_{i,i+2},\cdots),\; i\in\mN\}.$$
By incomparability, we need, for each $i<j$, to have $r_{i,j}>j$ with 
$$a_{i, r_{i,j}}> a_{j,r_{i,j}}.$$
But now, this implies
$$a_{0,r_{0,1}}> a_{1,r_{0,1}}\ge a_{1, r_{1,j}}>a_{2,r_{1,j}}$$
where we can take any $j\ge r_{0,1}$. Continuing like this, we create an infinite strictly decreasing chain (if we only consider the odd or only the even positions) of natural numbers, a contradiction.

Case (II): There is a bound on the number of labels that can be $\infty$ in every element of $S$. In that case, we can replace $S$ with an infinite subset such that every element has the same number, say $t\in\mN$, of labels equal to $\infty$. By definition of the inclusion order, we can actually remove the infinite part (the first $t$ labels of each element), and replace $S$ with an infinite set of incomparable partitions without $\infty$ labels. We can then replace all partitions by their transpose. Since Case (I) leads to a contradiction, we are again in Case (II). Applying the same procedure we just outlined again then produces an infinite set of incomparable {\em finite} partitions. That this is impossible is classical, see for instance the remark in \cite[\S3]{FL}.
\end{proof}

\begin{corollary}\label{Cor:Lem}
For $\lambda\in\bP_\infty$, let $\Phi_\lambda$ be the inductive system corresponding to all $D^\mu$ with $\mu\subset\lambda$. Every union
$\cup_{\lambda\in S}\Phi_\lambda$ is equal to $\cup_{\lambda\in S_0}\Phi_\lambda$ for a finite subset $S_0\subset S$.
\end{corollary}

\begin{proof}[Proof of Theorem~\ref{Prop:0}]Most results follow from Theorem~\ref{ThmZal}, which we use freely.

By replacing $D^\lambda$ by $\lambda$, we obtain a bijection between inductive systems and `saturated' ideals in the poset $\bY\subset\bP$. Here saturated means that removing any element produces a subset that is no longer an ideal.

For a non-empty inductive system $\Phi$, we can consider a path $\p$ in $\bY$ that is contained in $\Phi$ and set $\lambda:=\p(\infty)$. Then $\Phi_\lambda\subset\Phi$, so $\Phi$ is a union of such systems. For part (1) it suffices to show that the systems $\Phi_\lambda$ are indecomposable.

If a system $\Phi_\lambda$ were not indecomposable, by Corollary~\ref{Cor:Lem} we could write it as a \textit{finite} union of (smaller) systems $\Phi_\mu$, which one can easily disprove.

Part (2) now also follows from Corollary~\ref{Cor:Lem}.
\end{proof}

%

\begin{example}
We have
$$I_{(\infty^\infty)}\;=\; 0,\quad I_{(\infty,0^\infty)}=J_+,\quad\mbox{and}\quad I_{(1^\infty)}=J_-$$
and more generally
$$I_{\infty^m,n^\infty}\;=\;\mP_{m,n}=(J_+)^{\dagger m}\dagger (J_-)^{\dagger n},$$
as follows from follows from Example~\ref{ExDagger}(2) and \cite[Theorem~3.20]{Berele}.
\end{example}

\subsection{T-prime ideals}

%
%

In characteristic zero, the T-prime ideals have many alternative characterisations, see also Theorem~\ref{ThmSph} below.

\begin{theorem}\label{ThmClassT}
The following conditions are equivalent on a non-zero prime ideal $I<\bk S_\infty$.
\begin{enumerate}
\item $I$ is T-prime;
\item $I$ does not have any cover in the inclusion order on $\Spec \bk S_\infty$;
\item $I=\mP_{m,n}$ for $(m,n)\in \mN\times\mN$;
\item $I$ is not the (ordinary, i.e. $+$) sum of two other ideals;
\item $I=\TAnn(X)$ for some object $X$ in a tensor category $\cC$ over $\bk$.
\end{enumerate}
\end{theorem}
\begin{proof}
First we prove that (1) implies (3). 
Take an $\infty$-partition $\lambda$ not of the form $\lambda=(\infty^m, n^\infty)$ or $(\infty^\infty)$. We need to show that $I_\lambda$ is not T-prime. There are (unique) $m,n\in\mN$ (not both zero), so that
$$\lambda=(\infty^m,n+a_1,n+a_2,\ldots, n+a_r, n^\infty),$$
for a non-empty (finite) partition $(a_1,\ldots, a_r)$. We set
$$\mu:=((n+a_1)^{m+1},n+a_2,\ldots, n+a_r)\;\subset\;\lambda.$$
It follows from the Littlewood-Richardson rule that
$$\left[\Res^{S_{2d}}_{S_d\times S_d} D^\kappa \,:\, D^\mu\boxtimes D^\mu\right]\;=0,$$
whenever $\lambda\supset\kappa\vdash 2d$, with $d=|\mu|$. Indeed, in order to form the required LR tableau, the Yamanouchi word condition implies that we need to place $n+a_1$ labels `$m+1$' in the part of the skew tableau $\kappa\backslash \mu$ that can only be $n$ wide. This forces $a_1=0$, a contradiction.  Hence $\Phi^\lambda$ is not T-indecomposable, so $I_\lambda$ is not T-prime by Proposition~\ref{Prop:MIndec}.

That (3) implies (5) follows from Example~\ref{ExsVec}.

That (5) implies (1) is in Proposition~\ref{Lem:Tens}.

Equivalence between (2) and (3) follows from Theorem~\ref{Prop:0}(1). Indeed, $I_\lambda$ not having a cover is equivalent to there not being a partition $\nu\in \bP_\infty$ with the property that $\nu\subset\mu\subset\lambda$ implies either $\mu=\lambda$ or $\mu=\nu$. This simply means that $\lambda$ does not have `removable boxes', and the only such partitions are $(\infty^m, n^\infty)$ and $(\infty^\infty)$.

Finally, we prove equivalence between (3) and (4). By Theorem~\ref{ThmZal}, it suffices to prove that $\Phi_\lambda$ for an $\infty$-partition $\lambda$ can be written as an intersection of two inductive systems if and only if $\lambda$ is not of the form $(\infty^m,n^\infty)$. If $\lambda=(\infty^m,n^\infty)$, then every inductive system that strictly contains $\Phi_\lambda$ contains $((m+1)^{n+1})$. Conversely, if $\lambda$ is not of the form $(\infty^m,n^\infty)$, then it has $l>1$ addable boxes and we can write $\Phi_\lambda=\cap_\mu\Phi_\mu$ for $\mu$ ranging of the $l$ partitions obtained by adding one of those boxes.
\end{proof}

%
%

\subsection{Tensor categories}
Now assume that the field $\bk$ of characteristic zero is algebraically closed. By \cite{Del02}, any tensor category $\cC$ of moderate growth is a representation category of an affine group superscheme. Hence the ideal $\TAnn(X)$ of $X\in\cC$ is determined by the underlying supervector space of $X$.

\begin{theorem}\label{ThmsVec} Let $\cC$ be a tensor category of moderate growth and consider
$$\TAnn_{\cC}:\;M^{\oplus}(\cC)\to \TSpec\bk S_\infty.$$
\begin{enumerate}
\item $\TAnn_{\cC}$ factors via the Grothendieck semiring $M(\cC)$.
\item For $\cC=\sVec$ and $(m,n)\in\mN\times\mN$, we have
$$\TAnn([\bk^{m|n}])\;=\; \mP_{m,n}.$$
\item $\TAnn_{\cC}$ is injective if and only if $\cC=\Vecc$ or $\cC=\sVec$.
\item $\TAnn_{\cC}$ is an isomorphism of semirings if and only $\cC=\sVec$.
\end{enumerate}
\end{theorem}
\begin{proof}
Part (1) can easily be seen directly, but also follows from the fact that $\TAnn_{\cC}$ factors via the forgetful functor to $\sVec$. Yet another proof follows from Proposition~\ref{Lem:Tens}, since in characteristic zero $\Phi$ is a bijection.

Part (2) follows from Example~\ref{ExsVec}. Hence $\TAnn_{\sVec}$ is an isomorphism by 
Theorem~\ref{ThmClassT} and consequently $\TAnn_{\Vecc}$ is injective. It now also follows that (4) holds if we prove (3).

The remaining direction of part (3) is trivial for the case where $\cC$ is the representation category of an affine group scheme, since then any representation and a trivial representation of the same dimension are sent to the same T-prime ideal. The general case follows with minor additional work. Indeed, one can quickly observe that for $\TAnn_{\cC}$ to be injective, $\cC$ can only have one non-trivial simple object. The conclusion then follows via \cite[Corollaire~0.8]{Del02} and the observation that the only finite group with precisely two conjugacy classes of elements is $C_2$.
\end{proof}

\begin{corollary}\label{Cor0}
The semiring $\TSpec \bk S_\infty$ is isomorphic to $\mN[x]/(x^2-1)$, which is isomorphic to the semiring with underlying additive semigroup $\mN\times\mN$ with product
$$(m,n)(m',n')\;=\; (mm'+nn',mn'+nm').$$
The map $\gggr:\TSpec \bk S_\infty\to\mN$, is the semiring homomorphism $(m,n)\mapsto m+n$.
\end{corollary}

\section{The complex numbers}\label{SecComp}

When the base field is the field $\mC$ of complex numbers, there are interesting classes of simple unitary $S_\infty \times S_\infty$-representations. Prime ideals in $\bk(S_\infty\times S_\infty)$ restrict to prime ideals in $\bk S_\infty$, say for the left factor $S_\infty<S_\infty^{\times 2}$. Hence the left annihilator $\LAnn(V)<\bk S_\infty$ of such a representation~$V$ must be a prime ideal in the list in Theorem~\ref{Prop:0}(1), and similarly for the right annihilator. In this section we identify these ideals.

\subsection{Spherical representations}

Consider the {\bf Thoma simplex}
$$\Omega\;\subset\;[0,1]^{\mZ_{>0}}\times [0,1]^{\mZ_{>0}}$$
of pairs $\omega=(\alpha,\beta)$
satisfying
$$\alpha_1\ge\alpha_2\ge \alpha_3\ge\cdots, \quad \beta_1\ge \beta_2\ge \beta_3\ge\cdots,\quad\gamma(\alpha,\beta):=\sum_{i=1}^\infty (\alpha_i+\beta_i)\le 1.$$

By \cite[Theorem~7.20]{BO} there is a bijection
$\omega\mapsto  T^\omega,$
between $\Omega$ and the set of isomorphism classes of simple unitary {\bf spherical} $S_\infty\times S_\infty$-representations. Here `spherical' means there exists a spherical vector in the representation, which is a vector $v$ such that $gv=vg$ for all $g\in S_\infty$. Here, and below, we write the action of $g\in S_\infty=S_\infty\times \{e\}$ on the left, and denote the action of $h\in S_\infty= \{e\}\times S_\infty$ by writing $h^{-1}$ on the right.

\begin{theorem}\label{ThmSph}
Consider $\omega=(\alpha,\beta)\in \Omega$. If either $\gamma(\alpha,\beta)<1$ or $\alpha_i>0$ for all $i$ or $\beta_j>0$ for all $j$, then
$$\LAnn T^\omega\;=\;\RAnn T^\omega\;=\; 0.$$
In the remaining cases, 
$$\LAnn T^\omega\;=\;\RAnn T^\omega\;=\; \mathfrak{P}_{m,n},\quad\mbox{with}$$
$$m:=\max\{i\in\mN\mid \alpha_i\not=0\}\quad\mbox{and}\quad n:=\max\{i\in\mN\mid\beta_i\not=0\}.$$

\end{theorem}

\begin{remark}
Theorem~\ref{ThmSph} gives a negative answer to Zalesskii's question in \cite[Problem~5.12]{Za}, for the alternating group $A_\infty$. In fact, `$J_\chi=0$' (as defined in the proof of \ref{ThmSph} or \cite[Definition~5.8]{Za}) will be the case for `most' choices of extremal characters $\chi:A_\infty\to\mC$.
\end{remark}

Before we can give the proof of the theorem, we need to recall some definitions and results.

\subsubsection{} We follow the notion from \cite[Definition~1.7]{BO} of (extreme) characters $\chi: G\to\mC$ of groups $G$. For a finite group $G$ `extreme characters' are precisely the normalisation (such that $\chi(e)=1$) of usual (irreducible) characters, while `characters' now refer to positive linear combinations of usual (irreducible) characters that send the identity to $1$. We write $\chi^\lambda: S_n\to \mC$ for the function (extremal character) $\sigma\mapsto \chi(\sigma)/\chi(e)$, with $\chi$ the usual character of $D^\lambda$.

The set of characters $S_\infty\to\mC$ is in canonical bijection with the set of coherent systems of distributions on $\bY$, see \cite[Definition~3.4 and Proposition~3.5]{BO}, with an exchange of the `extremal' elements.

By Thoma's Theorem, see \cite[Theorems~3.12 and 3.15]{BO}, extreme characters of $S_\infty$ (or equivalently extreme coherent systems) are in bijection with elements of $\Omega$. We write the bijection as $\omega\mapsto \chi^\omega$. The representation $T^\omega$ is then characterised as follows: For a normalised spherical vector $v\in T^\omega$
\begin{equation}\label{eqsph}\chi^\omega(g)\;=\; (gv,v),\quad\mbox{ for all }\quad g\in S_\infty, \end{equation}
with $(\cdot,\cdot)$ the inner product of the Hilbert space $T^\omega$.

\subsubsection{}\label{limits} We will use an indirect description of the bijection $\omega\mapsto \chi^{\omega}$ mentioned above. Consider a path $\p$ in $\bY$. Although they are in general not well-defined, we can consider the pointwise limit function
\begin{equation}
\label{eqchi}\chi:=\lim_{n\to\infty} \chi^{\p(n)}\;:\; S_\infty\to\mC,
\end{equation}
and the limits
\begin{equation}\label{eqab}
\alpha_i\;:=\; \lim_{n\to\infty} \p(n)_i/n,\quad \beta_i\;:=\; \lim_{n\to\infty} (\p(n)^t)_i/n,\qquad i\in\mZ_{>0}.
\end{equation}
Vershik and Kerov proved in \cite[\S 5]{VK} that, for a given $\p$, the limit~\eqref{eqchi} exists if and only if the limits in~\eqref{eqab} exist. Moreover, we then have $\chi=\chi^\omega$, for $\omega:=(\alpha,\beta)\in \Omega$.

Moreover, as proved in \cite[Theorem~1]{VK}, every extremal character $\chi$ can be written as a limit \eqref{eqchi}, meaning the above completely describes the bijection. For our purposes, we need a slightly sharper statement, which can also be distilled from the work of Vershik and Kerov. To state it, we define the {\bf support} of a Thoma sequence $\omega\in \Omega$ as
the subset $\supp(\omega)\subset \bY$ of those $\lambda\vdash n$, for all $n$, such that $\chi^\lambda$ appears with non-zero coefficient in the expansion of $\chi^\omega|_{S_n}$ into (extremal) characters of $S_n$. The support corresponds to an inductive system, see \cite[Proposition~5.16]{Za}.

\begin{lemma}\label{Lem0}
Given $\omega\in \Omega$, there exists a path $\p$ in $\bY$ which is entirely contained in $\supp(\omega)$, so that $\omega$ is the limit obtained from $\p(n)$ as in \eqref{eqab}.
\end{lemma}
\begin{proof}
Following for instance \cite{VK}, the set $\cT(\bY)$ of paths on $\bY$ is a topological space by its inverse limit interpretation. Concretely, the open subsets are given by (finite) unions of cylinder subsets. A cyclinder subset of level $n$ is a subset comprising all paths that share a common initial segment of length~$n$.

Then there is a bijection between the set of coherent systems of distributions on $\bY$ and the set of Gibbs measures (called central Borel measures in \cite{VK}) on $\cT(\bY)$, see \cite[Proposition~7.11]{BO}. We describe one direction of the bijection. We start from a coherent system $\{P^{(n)}\}$, where $P^{(n)}$ is thus a probability measure on the finite set $\bY_n$. By definition of Borel measures, it is sufficient to determine the value of a Borel measure on the cylindrical sets. The Gibbs measure of the set corresponding to an initial segment
$$(\nu,\kappa,\ldots, \lambda\vdash n)$$ is given by $P^{(n)}(\lambda)$. 

Now, fix $\omega\in \Omega$, with corresponding $\chi^\omega$, coherent system $\{P^{(n)}\}$ and Gibbs measure $\mu^\omega$. By construction of $P^{(n)}$, we have $P^{(n)}(\lambda)\not=0$ if and only if $\lambda$ is in $\supp\omega$. It follows from the Vershik–Kerov ergodic theorem from \cite{VK}, see \cite[Corollary~7.19]{BO}, that $\mu^\omega$-almost every path $\p$ admits the limit \eqref{eqab}, producing $\omega$.

Now assume that the statement in the lemma is false. Then the set of paths with limit $\omega$ is contained in the union of cylinders corresponding to initial segments that end outside the support. By the above, this union has measure zero, contradicting the previous paragraph.
\end{proof}

\begin{remark}
The condition on $\p$ in Lemma~\ref{Lem0} is not automatic. For example $\p(n)=(n-1,1)$ is a path leading to the trivial character $\chi(\sigma)=1$ for all $\sigma\in S_\infty$. The lemma thus predicts that this $\chi$ can also be obtained as a limit from another path, and we can indeed take the path $\q(n)=(n)$.
\end{remark}

\begin{corollary}\label{CorSupp}
Consider $\omega=(\alpha,\beta)\in \Omega$. If either $\gamma(\alpha,\beta)<1$ or $\alpha_i>0$ for all $i$ or $\beta_j>0$ for all $j$, then
$\supp(\omega)=\bY$.
In the remaining cases, 
$$\supp(\omega)\;=\; \Phi^{(\infty^m,n^\infty)}$$
with $m,n$ as in Theorem~\ref{ThmSph}.
\end{corollary}
\begin{proof}
Assume first that the support of $\omega$ is not all of $\bY$. Then there are $s,t\in\mN$ so that $\mu_{s+1}\le t$ for all $\mu\in \supp(\omega)$. Let $\p$ be any path as in Lemma~\ref{Lem0}. It follows that
$$\sum_{i=1}^s\alpha_i+\sum_{j=1}^t\beta_j\;=\;\lim_{n\to\infty} \frac{\sum_{i=1}^s \p(n)_i+\sum_{j=1}^t (\p(n)^t)_j}{n}\;=\;1,$$
which proves the first claim. 

Now assume that $\gamma(\alpha,\beta)=1$ and $\alpha_i=0$ iff $i>m$ and $\beta_j=0$ iff $j>n$. By considering the obvious requirements for a path to lead to limit~ $\omega$ via~\eqref{eqab}, we can use Lemma~\ref{Lem0} to show that
$$\Phi^{(\infty^m,n^\infty)}\;\subset\; \supp(\omega).$$
To show equality, we can consider a path which, for $r\gg 0$,
is given by
$$\p(r)\;=\;(f(r),\lfloor r\alpha_2\rfloor,\ldots, \lfloor r\alpha_m\rfloor,\q(r)^t)$$
with
$$\q(r)\;=\;(\lfloor r\beta_1\rfloor-m,\ldots, \lfloor r\beta_n\rfloor -m),$$
and $f(r)\in\mN$ such that $\p(r)\vdash r$.
This path has limit $\omega$ under \eqref{eqab}, whence
$$\supp(\omega)\;\subset\; \Phi^{\p(\infty)}\;=\;\Phi^{(\infty^m,n^\infty)}.$$
This concludes the proof.
\end{proof}

\begin{proof}[Proof of Theorem~\ref{ThmSph}]
We can extend an extremal character $\chi$ to a $\mC$-linear function $\mC S_\infty\to\mC$ and set
$$J_\chi\;:=\;\{x\in \mC S_\infty\mid \chi(xg)=0,\; \forall g\in S_\infty\}.$$
Then, with $v$ the normalised spherical vector of $T^\omega$, and using \eqref{eqsph},
\begin{eqnarray*}
\LAnn T^\omega&=& \{x\in \mC S_\infty\mid (xgv, hv)=0,\;\;\forall g,h\in S_\infty\}\\
&=& \{x\in \mC S_\infty\mid (xgh^{-1}v,v)=0,\;\;\forall g,h\in S_\infty\}\;=\; J_{\chi^\omega}
\end{eqnarray*}
and similarly for the right annihilator.

Next we can observe that 
$J_{\chi^\omega}= I(\supp(\omega)), $
so that the conclusion follows from Corollary~\ref{CorSupp}.
\end{proof}

\subsection{Olshanski – Okounkov admissible representations}

Spherical representations belong to a broader class of `admissible' representations for $S_\infty\times S_\infty$. The classification of these representations was initiated by Olshanski and concluded by Okounkov in \cite{Ok}. Here, we determine their annihilator ideals, but we crucially use as input a result announced in \cite{VN} by Vershik and Nessonov without proof in the literature.

\subsubsection{}\label{SecTrOl} By a `Young distribution' of size $d\in\mN$, we refer to a function from $[-1,1]$ to $\bY$ such that the sum of all sizes of partitions in the image is $d$. In particular, for all but finitely many points in $[-1,1]$, the value is $\varnothing$.

By \cite[Theorem~3]{Ok}, the (isomorphism classes) of irreducible admissible representations are 
$V(\omega,\Lambda,M),$
where $\omega\in \Omega$, and $\Lambda$ and $M$ are two Young distributions of the same size $d$, called the depth of $V(\omega,\Lambda,M)$, such that
$$\begin{cases}\ell(\Lambda(x))+\ell(M(x))\;\le\; \sharp\{i\in\mZ_{>0}\mid \alpha_i=x\},&\; x>0,\\
\ell(\Lambda(x)^t)+\ell(M(x)^t)\;\le\; \sharp\{j\in\mZ_{>0}\mid \beta_j=-x\},&\; x<0.
\end{cases}$$

\subsubsection{}\label{ClaimVN} It was announced in \cite{VN} that the unitary representations of $\mC S_\infty$ given by (restriction to the left copy of $S_\infty$)
$$V(\omega,\Lambda,M)\quad\mbox{and}\quad\Ind^{S_\infty}_{S_r\times S_\infty[r]} (D^\lambda\boxtimes T^\omega)$$
are quasi-equivalent,
where $\lambda=\Lambda(0)$ and $r=|\lambda|$. \emph{Our proof of the second sentence in the following theorem is contingent on this claim.}

\begin{theorem}\label{ThmAdm} Consider an irreducible admissible representation $V(\omega,\Lambda,M)$ with $\omega=(\alpha,\beta)$.
If either $\gamma(\alpha,\beta)<1$ or $\alpha_i>0$ for all $i$ or $\beta_i>0$ for all $i$, then
$$\LAnn_{\bk S_\infty}V(\omega,\Lambda,M)\;=\;0\;=\;\RAnn_{\bk S_\infty}V(\omega,\Lambda,M).$$
In the remaining cases, set
$$m:=\max\{i\in\mN\mid \alpha_i\not=0\},\qquad n:=\max\{i\in\mN\mid\beta_i\not=0\},$$
$\lambda=\Lambda(0)$ and $\mu=M(0)$. Then
$$\LAnn_{\bk S_\infty} V(\omega,\Lambda,M)\;=\; I_{(\infty^m,n+\lambda_1,n+\lambda_2,\ldots)}\quad\mbox{and}$$
$$\RAnn_{\bk S_\infty} V(\omega,\Lambda,M)\;=\; I_{(\infty^m,n+\mu_1,n+\mu_2,\ldots)}.$$
\end{theorem}

\begin{remark}
Since there is no restriction on $\Lambda(0)$, Theorem~\ref{ThmAdm} implies that, contrary to spherical representations, the (left) annihilator ideals of admissible representations range over all prime ideals.
\end{remark}

We start the preparation for the proof. For $m\in\mZ_{>0}$, we denote by $S_\infty[m]<S_\infty$ the copy of $S_\infty$ that fixes every element of the set $\{1,2,\ldots, m\}$.

\begin{lemma}\label{Lem:LR}
For a partition $\nu\vdash j$ satisfying $\nu_{s+1}\le t$, for some $s,t\in\mN$, and partitions $\lambda\vdash r$ and $\kappa\vdash r+j$, the condition
$$[\Ind^{S_{r+j}}_{S_r\times S_j} D^\lambda\boxtimes D^\nu: D^\kappa]\;\not=\;0$$
implies that $\kappa_{s+i}\le t+\lambda_i $ for all $i>0$.
\end{lemma}
\begin{proof}
Immediate from the Littlewood-Richardson rule.
\end{proof}
For the remainder of the section, fix a triple $(\omega, \Lambda,M)$ as in \ref{SecTrOl}.

\begin{lemma}\label{LemSub}
We have 
$$\LAnn V(\omega,\Lambda,M)\;\subset\;\LAnn T^\omega\;\supset\; \RAnn V(\omega,\Lambda,M).$$
\end{lemma}
\begin{proof}
By construction of  $V(\omega,\Lambda,M)$, see for instance \cite[\S 1.2]{Ok}, its restriction to $S_\infty[d]\times S_\infty[d]$ has a direct summand isomorphic to $T^\omega$ (using the canonical identification $S_\infty[d]\simeq S_\infty$). Hence
$$\LAnn V(\omega,\Lambda,M)\cap \bk S_\infty[d]\;\subset\; I,$$
with $I$ the ideal $\LAnn T^\omega$ interpreted in $\bk S_\infty[d]$. The claim then follows by conjugation.
\end{proof}

\begin{proof}[Proof of Theorem~\ref{ThmAdm}]
By Lemma~\ref{LemSub} and Theorem~\ref{ThmSph}, we can focus on the final case. By \ref{ClaimVN}, the left annihilator of $ V(\omega,\Lambda,M)$ equals the annihilator of $\Ind^{S_\infty}_{S_r\times S_\infty[r]} (D^\lambda\boxtimes T^\omega)$.
By equation~\eqref{EqInd} and Theorem~\ref{ThmSph}, we then find
that 
$$\LAnn V(\omega,\Lambda,M)\;=\; \Ann \Ind^{S_\infty}_{S_r\times S_\infty[r]}(D^\lambda\boxtimes\varinjlim_i D^{\p(i)})$$
for any path $\p$ with $\p(\infty)=(\infty^m,n^\infty)$.
Moreover
$$\Ind^{S_\infty}_{S_r\times S_\infty[r]}(D^\lambda\boxtimes\varinjlim_i D^{\p(i)})\;\simeq\; \varinjlim_i \Ind^{S_{r+i}}_{S_r\times S_i}(D^\lambda\boxtimes D^{\p(i)}).$$
Hence, the inductive system corresponding to $\LAnn V(\omega,\Lambda,M)$ consists of all simple modules that appear in restrictions (to arbitrary $S_j<S_{r+m}$) of modules $\Ind^{S_{r+m}}_{S_r\times S_m}(D^\lambda\boxtimes D^{\mu})$ for $\mu\vdash m$ with $\mu_{m+1}\le n$. It follows quickly from Lemma~\ref{Lem:LR} that this system is precisely $\Phi_{(\infty^m,n+\lambda_1,n+\lambda_2,\ldots)}$, with notation as in Corollary~\ref{Cor:Lem}.
\end{proof}


\section{Positive characteristic}\label{SecTC}

In this section we consider an algebraically closed field $\bk$ of characteristic $p>0$. We investigate the annihilator map
$$\TAnn_{\cC}:\;M^{\oplus}(\cC)\;\to\;\TSpec\bk S_\infty$$
for tensor categories $\cC$ over $\bk$. We assume throughout that $\cC$ is a tensor category of moderate growth.

\subsection{Injectivity of the annihilator map}

\begin{prop}\label{PropInj}
\begin{enumerate}
\item If $\cC$ is Frobenius exact, then $\TAnn_{\cC}$ factors via $M^{\oplus}(\cC)\tto M(\cC)$.
\item If $\TAnn_{\cC}$ is injective, then for any non-split $\xi:\unit\to X$ in~$\cC$, there  is $n\in\mZ_{>0}$ for which $\xi^n:\unit\to\Sym^n X$ is zero.
\end{enumerate}
\end{prop}
\begin{proof}
Property (1) follows from the characterisation of Frobenius exactness in \cite[Theorem~C(iv)]{Tann}. Indeed, the latter states that any short exact sequence becomes split after applying a faithful symmetric monoidal functor. Alternatively, it follows from the main result of \cite{CEO}.

For part (2), consider the short exact sequence
$$0\to\unit\xrightarrow{\xi} X\to Y\to 0$$
defined by $\xi$. By assumption, the inclusion
$$\TAnn(X)\;\subset\; \TAnn(\unit\oplus Y)$$
is strict, and the conclusion follows from \cite[Corollary~3.2.6]{Tann}. 
\end{proof}

In characteristic zero, Theorem~\ref{ThmsVec}(3) shows that $\TAnn_{\cC}$ is injective if and only if $\cC$ is incompressible. By Proposition~\ref{PropInj}(2), a necessary condition for injectivity is a condition that is conjectured to be true for all known incompressible tensor categories ($\Ver_{p^\infty}$ and its tensor subcategories), and verified for $\Ver_{2^\infty}$ in \cite[Theorem~9.2.1]{Incomp}. This leads us to the following conjecture, in analogy with Theorem~\ref{ThmsVec}(3) and (4).
\begin{conjecture}\label{ConjVer}
The annihilator map
$$\TAnn_{\cC}:\;M^{\oplus}(\cC)\to\TSpec\bk S_\infty$$
is injective if and only if $\cC$ is incompressible, and a semiring isomorphism if and only if $\cC=\Ver_{p^\infty}$.  Moreover,
$$\TSys:\;M(\Ver_{p^\infty})\to\TSys\bk S_\infty$$
is a semiring isomorphism.
\end{conjecture}

\begin{remark}
\begin{enumerate}
\item A consequence of the conjecture would be that classifying (T-)prime ideals in $\bk S_\infty$ is a wild problem, while classifying T-indecomposable inductive systems is not.\color{black}
\item Conjecture~\ref{ConjVer} is a stronger version of \cite[Conjecture~1.4]{BEO}. Indeed, assume that Conjecture~\ref{ConjVer} is true and let $\cC$ be an arbitrary incompressible tensor category. By \cite[Corollary~5.2.8]{Incomp}, there exists an incompressible tensor category $\cD$ containing both $\Ver_{p^\infty}$ and $\cC$ as tensor subcategories. Then $\TAnn_{\Ver_{p^\infty}}$ factors via $\TAnn_{\cD}$. However, under Conjecture~\ref{ConjVer}, the former is an isomorphism, while the latter is injective, implying that $M^{\oplus}(\Ver_{p^\infty})\hookrightarrow M^{\oplus}(\cD)$ is an isomorphism, forcing $\cD=\Ver_{p^\infty}$ and hence $\cC\subset \Ver_{p^\infty}$. The conclusion follows from \cite[Theorem~A]{Incomp}.
\end{enumerate} 
\end{remark}

While Conjecture~\ref{ConjVer} seems currently out of reach, the following special case is more tractable. It is an analogue of Theorem~\ref{ThmsVec}(3), but the proof of the latter fails for $p>3$. By \cite{CEO} it suffices to show that for a semisimple tensor category $\cC$ of moderate growth, the unique tensor functor $\cC\to\Ver_p$ is the inclusion of a subcategory if $K_0(\cC)\to K_0(\Ver_p)$ is injective.
\begin{conjecture}\label{ConjVer2}
Let $\cC$ be Frobenius exact.
If the annihilator map
$$\TAnn_{\cC}:\;M^{\oplus}(\cC)\to\TSpec\bk S_\infty$$
is injective, then $\cC$ is a tensor subcategory of $\Ver_p$ (and thus incompressible).
\end{conjecture}

\subsection{Symmetrisers in ideals}
To establish the first steps towards Conjecture~\ref{ConjVer}, we introduce invariant of ideals, complementary to the ones in Subsection~\ref{SecGrowth}.

\subsubsection{}\label{as} Denote the symmetriser and skew symmetriser in $\bk S_n$ by $s_n$ and $a_n$. Let $I<\bk S_\infty$ be a non-zero ideal. It follows from \cite[Proposition~A.1]{CEO}, see also the proof of \cite[Lemma~4.10]{CEO}, that $s_n\in I$ for some $n$. We can thus define
$$\sss(I)\;:=\;\max\{n\in\mZ_{>0}\mid s_n\not\in I\}\;\in\;\mN$$
and similarly
$$\aaa(I)\;:=\;\max\{n\in\mZ_{>0}\mid a_n\not\in I\}\;\in\;\mN.$$
By convention, $\sss(\bk S_\infty)=0= \aaa(\bk S_\infty).$

\begin{example}
For the augmentation ideal $J_+$ we have
$$\aaa(J_+)=1\quad\mbox{and}\quad\sss(J_+)=p-1.$$
\end{example}

\begin{lemma}
For non-zero ideals $I,J$ in $\bk S_\infty$, we have
$$\sss(I\dagger J)\;=\; \sss(I)+\sss(J)\quad\mbox{and}\quad \aaa(I\dagger J)\;=\;\aaa(I)+\aaa(J).$$
\end{lemma}
\begin{proof}
We write the proof for symmetrisers. We can observe that, for $r,t\in\mZ_{>0}$, we have $s_{r+t}=s'(s_r\bullet s_t)$, for $s'$ a sum of shortest coset representatives of $S_{r+t}/(S_r\times S_t)$. Hence, if $s_m\not\in I$ and $s_n\not\in J$, it follows that $s_{m+n}$ does not act trivially on the generator $1\otimes 1\otimes 1$ of $\Ind^{S_{m+n}}_{S_m\times S_n}(\bk S_m/I_m\boxtimes \bk S_n/J_n)$. Hence we find
$$\sss(I\dagger J)\;\ge\; \sss(I)+\sss(J).$$
To get the inequality in the other direction, we can use additionally that the symmetriser is in the centre of the group algebra. We can observe that $s_{m+1}\in I$ and $s_{n+1}\in J$ imply that $s_{m+n+1}$ acts trivially on 
$$\Ind^{S_{m+n+1}}_{S_i\times S_j}(\bk S_i/I_i\boxtimes \bk S_j/J_j)$$ for $i+j=m+n+1$, since either $i\ge m+1$ or $j\ge n+1$.
\end{proof}

\begin{lemma}\label{lemfin}
For any $I\in \TSpec\bk S_\infty$, we have $\gggr(I)\in\mR_{\ge 0}$.
\end{lemma}
\begin{proof}
As recalled above, we have $s_{n+1}\in I$, for some $n$. Hence $ \TAnn(V)\subset I$, for a vector space $V$ of dimension $n$, by classical invariant theory. Thus
$$\gggr(I)\;\le\; \gggr(\TAnn(V))\;=\, n,$$
see Lemma~\ref{Lemmag}(1).
\end{proof}

\subsection{The annihilator map for Verlinde categories}
Conjecture~\ref{ConjVer} predicts that $\TAnn_{\Ver_{p^n}}$ is injective. Here we verify some small cases.
\begin{theorem}\label{ThmVer}
\begin{enumerate}
\item The annihilator map $\TAnn_{\Ver_4}$
is injective.
\item For every prime $p$, the annihilator map
$\TAnn_{\Ver_p}$
is injective.
\end{enumerate}
\end{theorem}
\begin{proof}
First we prove part (1). There are three indecomposable objects in $\Ver_4$, see \cite{BE}, the unit $\unit$, a simple projective $V$ and the projective cover $P$ of $\unit$. We will use
\begin{equation}\label{eq:Ver4}V^{\otimes 2}\simeq P\quad\mbox{and}\quad V\otimes P\simeq V^2.\end{equation} We associate three invariants to the annihilator of an object $X$ in $\Ver_4$. With $K:=\TAnn(V)$, set $x_1(X):=\sss(\TAnn(X))$ and
$$x_2(X):=\sss(K\Join\TAnn(X))\quad\mbox{and}\quad x_3(X):=\sss(K\Join K\Join \TAnn(X)).$$ 
In practice, we can use $K\Join\TAnn(X)=\TAnn(V\otimes X)$.

If
$X$ is $\unit^l\oplus V^m\oplus P^n,$
then by \cite[Example~4.6]{CEO} and \eqref{eq:Ver4}
$$x_1(X)= l+2m+4n,\quad x_2(X)=2l+4m+4n,\quad x_3(X)=4l+4m+8n.$$
It now follows from the observation
$$\det\left(\begin{array}{ccc}
1&2&4\\
2&4&4\\
4&4&8
\end{array}\right)\;=\;-16\;\not=\; 0$$
that $\TAnn_{\Ver_4}$ is injective.

For part (2), we consider $\Ver_p$ as the semisimplification of $\mathrm{Rep} C_p$, so that the simple $L_i\in \Ver_p$, for $0<i<p$, is the image of the indecomposable $C_p$-representation of dimension $i$, see \cite{Os}. As in part (1), for $X\in\Ver_p$ we define $x_i(X):=\aaa(\TAnn(L_2^{\otimes i-1}\otimes X))$, for $i\in\mZ_{>0}$, which again only depend on $\TAnn(X)$. For $W\in\mathrm{Rep} C_p$, with image $\overline{W}\in\Ver_p$, the value $\aaa(\TAnn(\overline{W}))$ is the (vector space) dimension of $W$ minus the dimension of the maximal projective summand.
For $i+j\le p$, we thus have
$x_i(L_j)=2^{i-1}j$ and for $i+j=p+1$, we have
$x_i(L_j)=2^{i-1}j-p$. Again it suffices to show that the determinant of the $(p-1)\times(p-1)$-matrix $(x_i(L_j))_{1\le i,j\le p-1}$ is non-zero. After subtracting the first row $2^{i-1}$ times from the $i$-th row, we obtain a matrix $A$ with 
$A_{ij}=0$ if $i+j<p$ and $i>1$, and $A_{ij}=-p$ if $i+j=p+1$, hence
$$\det((x_i(L_j))_{ij})\;=\; \det(A)\;=\; (-1)^{\frac{(p+1)(p-2)}{2}}p^{p-2}\;\not=\; 0,$$
as desired.
\end{proof}

\begin{remark}
The method in the proof of Theorem~\ref{ThmVer} would extend to the essential image of the defining monoidal functor
$$F:\mathrm{Tilt} SL_2\to\Ver_{p^n},$$
see \cite{BEO, AbEnv}, if it remains true that 
$$\aaa(\TAnn_{\Ver_{p^n}}(FT_i))=\dim_{\bk}T_i,\quad\mbox{ for }\quad0\le i\le p^n-2,$$ where $T_i$ is the $SL_2$-tilting module of highest weight $i$. The cases dealt with in Theorem~\ref{ThmVer} are precisely those for which $F$ is essentially surjective.
\end{remark}

We conclude with an observation regarding the growth rate of ideals corresponding to $\Ver_p$.

\begin{prop}\label{Propgdg}Let $p$ be an odd prime and let $J_1,J_2,\ldots, J_{p-1}$ be the exhaustive list of maximal ideals in $\bk S_\infty$ from \cite{BK}. Then
$$\gggr(J_j)\;=\; \sin\left(\frac{\pi j }{p}\right)/\sin \left(\frac{\pi}{p}\right),\quad\mbox{for}\quad 1\le j<p.$$
Equivalently, for any simple object $V$ in $\Ver_p$, we have
$$\gggr(\TAnn(V))\;=\;\mathtt{gd}(V).$$
\end{prop}
\begin{proof}
Firstly, by \cite[Lemma~8.3]{CEO}, the growth dimension of $V$ equals its Frobenius-Perron dimension, so $\mathtt{gd}(L_j)=\sin\left(\frac{\pi j }{p}\right)/\sin \left(\frac{\pi}{p}\right)$, see for instance \cite{BEO}, and it suffices to prove the second claim. 
For that claim, we already know that $\gggr(\TAnn(V)) \le \mathtt{gd}(V)$, by Remark~\ref{ggd}.

By \cite[Lemma~8.5]{CEO}, it is sufficient to prove $\gggr(\TAnn(L_j))\ge\mathtt{gd}(L_j)$ for $L_j$ regarded as an object in the fusion category $\Ver_p(SL_j)$, as done in the proof of \cite[Theorem~4.1.1]{PolyFun}. The advantage there is that, as observed {\it loc. cit.}, the braiding homomorphism in $\Ver_p(SL_j)$
$$\bk S_n\;\to\; \End(L_j^{\otimes n})$$
is surjective. Now, let $d$ be the number of isomorphism classes of simple objects in $\Ver_p(SL_j)$. We then find
$$\ell(L_j^{\otimes n})^2\;\le\; d\,\dim\End(L_j^{\otimes n})\;=\;d\,\dim (\bk S_n/\TAnn(L_j)_n).$$
Taking the limit of the $2n$-th roots shows $\mathtt{gd}(L_j)\le \gggr(\TAnn(L_j))$.
\end{proof}

\begin{question}
Does the second equality in Proposition~\ref{Propgdg} remain valid for all objects in  $\Ver_p$? Equivalently, is $\gggr$ additive on the submonoid of $(\TSpec\bk S_\infty,\dagger)$ generated by the maximal ideals. 
\end{question}


\section{Categorical dimension of ideals}\label{SecDim}Let $\bk$ again be an arbitrary field.
In the current section, we broaden our scope to $\bk$-linear symmetric monoidal categories more general than tensor categories.

\subsection{Beyond tensor categories} Let $\cA$ be a $\bk$-linear symmetric monoidal category and consider $X\in\cA$.

\subsubsection{}\label{nonrig} For any $n\in\mZ_{>0}$, we can consider the ideal $I_n<\bk S_n$, defined as the kernel of the braiding morphism $\bk S_n\to\End(X^{\otimes n})$. However, in this generality, we only have an inclusion $I_n\subset I_{n+1}\cap \bk S_n$, rather than an equality. Nonetheless, we can still define an ideal $I=\cup_{n}I_n$ in $\bk S_\infty$, generalising the $\TAnn$ construction from tensor categories. 

Contrary to the case of tensor categories, which relates to the very restrictive class of T-prime ideals, any ideal in $\bk S_\infty$ can be realised in the current generality. More precisely, any collection of ideals  $I_n<\bk S_n$ with $I_n\subset I_{n+1}\cap \bk S_n$ can be realised. It suffices to look at the universal case $\mathcal{S}ym$, see \cite[\S4]{CEO}, and let such a collection $I_n$ define a tensor ideal.

Moreover, taking the category of presheaves $\cA^{\mathrm{op}}\to\Vecc$ with Day convolution, shows that restricting to non-rigid abelian symmetric monoidal categories leads to the same behaviour.

\subsubsection{} For a more interesting situation, we keep focusing on rigid tensor categories, but do not restrict to abelian categories. Concretely, we will say that a $\bk$-linear symmetric category
$(\cA,\otimes,\unit)$ is an {\bf RSM category} (rigid symmetric monoidal) if
 $\bk\to\End(\unit)$ is an isomorphism and $(\cA,\otimes,\unit)$ is rigid. 

As in \ref{nonrig}, an object $X\in \cA$ in an RSM category leads to ideals $I_n<\bk S_n$, but now $I_n=I_{n+1}\cap \bk S_n$. Indeed, if $a\bullet 1\in I_{n+1}$ for some $a\in \bk S_n$, it suffices to apply $-\otimes X^\ast$ followed by strategically placing an evaluation and coevaluation to recover $a$, showing that $a\in I_n$. We use again the notation $\TAnn_{\cA}(X)<\bk S_\infty$ for the resulting ideal.

By Theorem~\ref{ThmFinal} below, in characteristic zero, this leads again precisely to the class of T-prime ideals, whereas in positive characteristic we can get ideals that are not T-prime. In particular, the ideals in $\bk S_\infty$ we get from RSM categories are the same as from tensor categories in characteristic zero, but more general in positive characteristic.

\subsection{Categorical dimension of an ideal}

\subsubsection{} 
For a fixed $\delta\in\bk$, we define a $\bk$-linear map $R_\delta:\bk S_{n+1}\to \bk S_n$. $R_\delta$ sends $\sigma\in S_{n+1}$ to $\delta\sigma\in \bk S_n$ if $\sigma\in S_{n}<S_{n+1}$. If $\sigma\in S_{n+1}\backslash S_n$, then $R(\sigma)$ is the element of $S_n$ obtained by simply removing `$n+1$' from an expression of $\sigma$ as a product of disjoint cycles. If we interpret $\sigma \in S_{n+1}$ as a diagram by connecting each of $n+1$ elements on one horizontal line with precisely one on a lower horizontal line, $R_\delta(\sigma)$ is obtained by simplifying the diagram
\begin{equation}
\label{defR}
   R_\delta(\sigma)\;:=\; \begin{array}{c}
      \begin{tikzpicture}[xscale=.6,yscale=1]
        \begin{scope}
          \clip (-1.8,1) rectangle (2.5,-1);
  \draw(1.4,.6) to[out=90,in=180] (1.5,.8)
  to[out=0,in=90] (2,0) to[out=-90,in=0] (1.5,-.8) to[out=180,in=-90]
  (1.4,-.6) to (1,.6);
  \draw[fill=white] (-1.4,.6) rectangle (1.4,-.6);
  \node at (0,0) {$\sigma$};
  \end{scope}
\end{tikzpicture}
      \end{array}
\end{equation}
and multiplying with $\delta$ if a closed loop occurs.

\begin{definition} A proper non-zero ideal $I<\bk S_\infty$ {\bf admits dimension $\delta\in \bk$} if $R_\delta(I_{n+1})\subset I_n$ for all $n>0$.
\end{definition}

\begin{lemma}\label{LemDim}
If a proper non-zero ideal $I<\bk S_\infty$ admits dimension $\delta\in \bk$, then $\delta$ is unique and is in the image of the ring homomorphism $\mZ\to\bk$.
\end{lemma}
\begin{proof}
Assume $\mathrm{char}(\bk)=p>0$. 
As recalled in \ref{as}, then $a_{n+1}\in I$ while $a_{n}\not \in I$ for some $n\in\mZ_{>0}$. It follows quickly that
$$I_n\,\ni \,R_\sigma(a_{n+1})\;=\; (\delta-n)\,a_n$$
and hence $\delta$ is the image of $n$ under $\mZ\to\bk$. 

The characteristic~0 case follows from Theorem~\ref{ThmFinal}(3) below.
\end{proof}

\subsection{Annihilator ideals for RMS categories}

\subsubsection{The oriented Brauer category} The oriented Brauer category $\OB(\delta)$ for a parameter $\delta\in \bk$, see for instance \cite{BCNR, Del07}, is the universal RSM category on a generating object $V_\delta$ of categorical dimension $\delta$. It has a simple diagrammatic description, where closed loops are set to $\delta$. In particular, $\End(V_\delta^{\otimes n})=\bk S_n$. The universality means that for any RMS category $\cA$ with an object $X$ of categorical dimension $\delta$, there is a unique (up to isomorphism) $\bk$-linear symmetric monoidal functor $\OB(\delta)\to \cA$ with $V_\delta\mapsto X$.

We refer to \cite{Selecta} for the basic notions of tensor ideals in monoidal categories.

\begin{theorem}\label{ThmFinal}
\begin{enumerate}
\item A non-zero proper ideal $I<\bk S_\infty$ admits a dimension if and only if it is of the form $\TAnn_{\cA}(X)$ for an object $X$ in a RSM category $\cA$. The dimension of $I$ then equals the categorical dimension of $X$.
\item For every $\delta\in\bk$, there is a bijection between the non-zero proper tensor ideals in $\OB(\delta)$ and the proper non-zero ideals in $\bk S_\infty$ admitting dimension $\delta$, given by $\cI\mapsto \TAnn_{\OB(\delta)/\cI}(V_\delta)$.
\item If $\mathrm{char}(\bk)=0$, then a non-zero proper ideal $I<\bk S_\infty$ admits a dimension if and only if it is T-prime. Moreover, the dimension of $\mP_{m,n}$ is $m-n$.
\item If $\mathrm{char}(\bk)>3$, there are ideals in $\bk S_\infty$ that admit a dimension, but are not T-prime.
\end{enumerate}
\end{theorem}
\begin{proof}
We start by proving (2). By adjunction, all information of a tensor ideal $\cI$ in $\OB(\delta)$ is contained in $I:=\TAnn_{\OB(\delta)/\cI}(V_\delta)$, where in particular $I_n=\cI(V_\delta^{\otimes n},V_\delta^{\otimes n})$. It thus only remains to prove that an ideal $I<\bk S_\infty$ can be realised inside $\OB(\delta)$ if and only if it admits dimension $\delta$. One direction is immediate, since the operation in \eqref{defR} can be interpreted inside $\OB(\delta)$ as tensoring with $V_\delta^\ast$ followed by pre- and post-composing with morphism. Hence an ideal must admit dimension $\delta$ in order to correspond to a tensor ideal in $\OB(\delta)$. The other direction follows from a slightly expanded diagrammatic interpretation.

Part (1) then follows from part (2) and the universality of $\OB(\delta)$.

For part (3) we use parts (1) and (2). Then we can invoke the classification of tensor ideals in $\OB(\delta)$ from \cite{Selecta}, stating that they are given by (for $\delta\in\mZ$, and there are no proper non-zero tensor ideals if $\delta\not\in\mZ$) the kernels of the symmetric monoidal functors $\OB(\delta)\to\sVec$ sending $V_\delta$ to $\bk^{m|n}$ for all $m-n=\delta$. Now compare with Theorem~\ref{ThmsVec}(2).

Finally, for part (4) it suffices to observe that, as proved in \cite[Lemma~7.1.3]{CEOq}, there exists a tensor ideal $\cI$ in $\OB(3)$ with the property that there are morphisms $f,g$ in $\OB(3)$ not in $\cI$ but for which $f\otimes g\in \cI$. Indeed, by adjunction we can assume that $f $ and $g$ are endomorphisms of tensor powers of $V_\delta$, so that the conclusion follows from part (2).
\end{proof}

\begin{corollary}\label{CorFin}
\begin{enumerate}
\item If a skew symmetriser vanishes on an object in a RSM category, then its categorical dimension corresponds to the last one that does not vanish.
\item If $\mathrm{char}(\bk)=p>0$, then $\OB(\delta)$ only has non-zero proper tensor ideals if $\delta\in \mF_p$.
\end{enumerate}
\end{corollary}
\begin{proof}
Part (1), follows from the proof of Lemma~\ref{LemDim} and Theorem~\ref{ThmFinal}(1). Part (2) follows from Theorem~\ref{ThmFinal}(2) and Lemma~\ref{LemDim}.
\end{proof}

Corollary~\ref{CorFin}(2) is analogous with the characteristic zero result, but more remarkable, because in positive characteristic, $\OB(\delta)$ for $\delta\in \mF_p$ is not semisimple abelian (after taking the additive and idempotent closure). Another proof of Corollary~\ref{CorFin}(2) goes as follows. If $\delta\not\in \mF_p$, then $\OB(\delta)$ has a nilpotent endomorphism, $1-\sigma$ acting on $V_\delta^{\otimes p}$ for $\sigma\in S_p$ a $p$-cycle, of non-zero trace. Hence the proof of \cite[Corollary~1.5]{EP}, or the combination of \cite[Corollary~1.5]{EP} and the observation in \ref{as}, show that $\OB(\delta)$ has no proper non-zero tensor ideals.

Conjecture~\ref{ConjVer} predicts the following (for $\mathrm{char}(\bk)=p>0$).

\begin{conjecture} Every T-prime ideal admits a dimension. \end{conjecture}

\subsection*{Acknowledgement} The research was supported by ARC grant FT220100125. The author thanks Pavel Etingof, Alexander Kleshchev and Nikolay Nessonov for helpful correspondence.



\begin{thebibliography}
	{DMNO}
	
	\bibitem[BK]{BK} A.~Baranov, A.~Kleshchev:
Maximal ideals in modular group algebras of the finitary symmetric and alternating groups.
Trans. Amer. Math. Soc. 351 (1999), no.2, 595--617.


\bibitem[Be]{Be} D. Benson, Commutative Banach algebras and modular representation theory. Mem. Amer. Math. Soc. 298 (2024), no. 1488.

\bibitem[BE]{BE} D.~Benson, P.~Etingof:
Symmetric tensor categories in characteristic 2.
Adv. Math. 351 (2019), 967--999.


\bibitem[BEO]{BEO} D.~Benson, P.~Etingof, V.~Ostrik: New incompressible symmetric tensor categories in positive characteristic. Duke Math. J. 172 (2023), no. 1, 105--200.

\bibitem[BR]{Berele}
A.~Berele, A.~Regev: Hook Young diagrams with applications to combinatorics and to representations of Lie superalgebras. Adv. in Math. {\bf64} (1987), no. 2, 118--175.

	
	\bibitem[BO]{BO} A.~Borodin, G.~Olshanski:
Representations of the infinite symmetric group.
Cambridge Stud. Adv. Math., 160
Cambridge University Press, 2017. 

\bibitem[Br]{Br} O.~Bratteli:
Inductive limits of finite dimensional C*-algebras.
Trans. Amer. Math. Soc.   171 (1972), 195--234.


\bibitem[BCNR]{BCNR} J.~Brundan, J.~Comes, D.~Nash, A.~Reynolds:
A basis theorem for the affine oriented Brauer category and its cyclotomic quotients.
Quantum Topol. 8 (2017), no. 1, 75--112.


\bibitem[Co1]{Selecta} K.~Coulembier: Tensor ideals, Deligne categories and invariant theory. Selecta Math. (N.S.) 24 (2018), no. 5, 4659--4710.

\bibitem[Co2]{Tann} K.~Coulembier: Tannakian categories in positive characteristic. Duke Math. J. 169 (2020), no. 16, 3167-–3219.

\bibitem[Co3]{AbEnv}  K.~Coulembier: Monoidal abelian envelopes. Compos. Math. 157 (2021), no. 7, 1584--1609.

\bibitem[Co4]{PolyFun} K.~Coulembier: Inductive systems of the symmetric group, polynomial functors and tensor categories. arXiv:2406.00892. To appear in Annales de l'ENS.

\bibitem[CEKO]{CEKO} K.~Coulembier, P.~Etingof, A.~Kleshchev, V.~Ostrik:
Super invariant theory in positive characteristic.
Eur. J. Math. 9 (2023), no. 4, Paper No. 94, 39 pp.

\bibitem[CEO1]{CEO} K.~Coulembier, P.~Etingof, V.~Ostrik: On Frobenius exact symmetric tensor categories. With an appendix by A.~Kleshchev. Ann. of Math. (2) 197 (2023), no. 3, 1235--1279.

\bibitem[CEO2]{Incomp} K.~Coulembier, P.~Etingof, V.~Ostrik: Incompressible tensor categories. Adv. Math. 457 (2024), Paper No. 109935.

\bibitem[CEO3]{CEOq} K.~Coulembier, P.~Etingof, V.~Ostrik: Tensor ideals of abelian type and quantum groups. arXiv:2511.08859.

\bibitem[COT]{COT} K.~Coulembier, V.~Ostrik, D.~Tubbenhauer:
Growth rates of the number of indecomposable summands in tensor powers.
Algebr. Represent. Theory 27 (2024), no. 2, 1033--1062.



	\bibitem[De1]{Del02} P.~Deligne: Cat\'egories tensorielles. Mosc. Math. J. 2 (2002), no. 2, 227--248.
	
	
\bibitem[De2]{Del07} P.~Deligne:
La catégorie des représentations du groupe symétrique $S_t$, lorsque $t$ n'est pas un entier naturel. Algebraic groups and homogeneous spaces, 209--273.
Tata Inst. Fund. Res. Stud. Math., 19
Published for the Tata Institute of Fundamental Research, Mumbai, 2007.
	
	
	
		
\bibitem[EGNO]{EGNO}P.~Etingof, S.~Gelaki, D.~Nikshych, V.~Ostrik:
Tensor categories. 
Mathematical Surveys and Monographs, 205. American Mathematical Society, Providence, RI, 2015. 	

\bibitem[EO]{EO} P. Etingof, V. Ostrik: On the Frobenius functor for symmetric tensor categories in positive characteristic.  J. Reine Angew. Math. 773 (2021), 165--198.


\bibitem[EP]{EP} P. Etingof, D. Penneys: Rigidity of non-negligible objects of moderate growth in braided categories. arXiv:2412.17681.

\bibitem[Fo]{For} E.~Formanek:
A problem of Passman on semisimplicity.
Bull. London Math. Soc. 4 (1972), 375--376.


\bibitem[FL]{FL} E.~Formanek, J.~Lawrence:
The group algebra of the infinite symmetric group.
Israel J. Math.   23 (1976), no. 3--4, 325--331.
	
	
	\bibitem[Ok]{Ok} A.~Okounkov:
On representations of the infinite symmetric group. 
J. Math. Sci. (New York)   96 (1999), no. 5, 3550--3589.



\bibitem[Os]{Os} V.~Ostrik: On symmetric fusion categories in positive characteristic. Selecta Math. (N.S.) 26 (2020), no. 3, Paper No. 36, 19 pp.


\bibitem[VK]{VK}A.M.~Vershik, S.V.~Kerov:
Asymptotic theory of the characters of a symmetric group.
Funktsional. Anal. i Prilozhen.   15 (1981), no. 4, 15--27, 96.


\bibitem[VN]{VN}A.M.~Vershik, N.I.~Nessonov:
Stable representations of the infinite symmetric group.
Izv. Math.   79 (2015), no. 6, 1184--1214.


	\bibitem[Za1]{Za} A.E.~Zalesskiĭ:
Group rings of simple locally finite groups.
NATO Adv. Sci. Inst. Ser. C: Math. Phys. Sci., 471 (1995).


\bibitem[Za2]{Za2} A.E.~Zalesski:
Modular group rings of the finitary symmetric group.
Israel J. Math. 96 (1996), 609--621.
	
	\end{thebibliography}
\end{document}